%% file: qcolorings.tex
\documentclass[11pt]{article}

\usepackage{amsthm}
\usepackage{amsmath}
\usepackage{amssymb}

\usepackage{graphicx}
\usepackage{tikz}
\usepackage{caption}
\usepackage{float}

\usepackage{enumerate}
\usepackage{framed}
\usepackage{hyperref}
\usepackage{xspace}

\newtheorem{theorem}{Theorem}[section]
\newtheorem{proposition}[theorem]{Proposition}
\newtheorem{lemma}[theorem]{Lemma}

\newtheorem{corollary}[theorem]{Corollary}
\newtheorem{observation}[theorem]{Observation}

\newtheorem{conjecture}[theorem]{Conjecture}

\newcommand{\secref}[1]{Section~\ref{#1}\xspace}
\newcommand{\thmref}[1]{Theorem~\ref{#1}\xspace}
\newcommand{\lemref}[1]{Lemma~\ref{#1}\xspace}
\newcommand{\propref}[1]{Proposition~\ref{#1}\xspace}
\newcommand{\corref}[1]{Corollary~\ref{#1}\xspace}

\newcommand{\obsref}[1]{Observation~\ref{#1}\xspace}

\newcommand{\figref}[1]{Figure~\ref{#1}\xspace}
\newcommand{\conjref}[1]{Conjecture~\ref{#1}\xspace}

\newif\ifnotescoauthor
\notescoauthortrue 
\newcommand{\comment}[2]{\ifnotescoauthor \begin{framed} $\blacktriangleright$
{\sf #1} $\blacktriangleleft$ \ifx&#2&\else\medskip\hrule\medskip
#2\fi\end{framed}\fi}

\newcommand{\R}{{\mathbb{R}}}
\newcommand{\OBJ}{{\textsc{OBJ}_q}}
\newcommand{\OBJq}[1]{{\textsc{OBJ}_{#1}}}
\newcommand{\V}{{\textsc{V}_q}}
\newcommand{\Vq}[1]{{\textsc{V}_{#1}}}
\newcommand{\E}{{\textsc{E}_q}}
\newcommand{\Eq}[1]{{\textsc{E}_{#1}}}
\newcommand{\FEAS}{{\textsc{FEAS}_q}}
\newcommand{\FEASq}[1]{{\textsc{FEAS}_{#1}}}
\newcommand{\OPT}{{\textsc{OPT}_q}}
\newcommand{\OPTq}[1]{{\textsc{OPT}_{#1}}}
\newcommand{\OPTnone}{{\textsc{OPT}}}

\newcommand{\valpha}{{\boldsymbol \alpha}}
\newcommand{\vbeta}{{\boldsymbol \beta}}
\newcommand{\vgamma}{{\boldsymbol \gamma}}

\newcommand{\vA}{{\mathbf{A}}}

\newcommand{\vx}{{\mathbf{x}}}

\newcommand{\OPTprob}{{\textsc{Main Optimization Problem}}}
\newcommand{\SUPP}{{\textsc{SUPP}_q}}
\newcommand{\cei}[1]{{\left\lceil{#1}\right\rceil}}
\newcommand{\flo}[1]{{\left\lfloor{#1}\right\rfloor}}

\newcommand{\Step}[1][Step]{\underline{\emph{#1:}}\enskip}
\newcommand{\Text}[1]{\enskip\text{#1}\enskip}
\newcommand{\lt}{\left}
\newcommand{\rt}{\right}
\newcommand{\eps}{\varepsilon}
\newcommand{\mc}[1]{\mathcal #1}

\title{Maximizing proper colorings on graphs}
\author{Jie Ma
\thanks{Department of Mathematical Sciences,
Carnegie Mellon University, Pittsburgh, PA 15213, USA.
Email: \texttt{jiemath@andrew.cmu.edu}.} \and
Humberto Naves
\thanks{Institute for Mathematics and its Applications,
University of Minnesota, Minneapolis, MN 55455, USA.
Email: \texttt{hnaves@ima.umn.edu}.
This research was supported in part by the Institute for
Mathematics and its Applications with funds provided by
the National Science Foundation.}}

\date{}

\begin{document}
\maketitle
\setcounter{page}{1}

\vspace{-2em}
\begin{abstract}

  The number of proper $q$-colorings of a graph $G$, denoted by $P_G(q)$,
  is an important graph parameter that plays fundamental role in graph theory,
  computational complexity theory and other related fields.
  We study an old problem of Linial and Wilf to find the graphs with
  $n$ vertices and $m$ edges which maximize this parameter.
  This problem has attracted much research interest in recent years,
  however little is known for general $m,n,q$. Using analytic and
  combinatorial methods, we characterize the asymptotic structure of extremal
  graphs for fixed edge density and $q$. Moreover, we disprove a
  conjecture of Lazebnik, which states that the Tur\'{a}n graph $T_s(n)$
  has more $q$-colorings than any other graph with the same number of
  vertices and edges. Indeed, we show that there are infinite many
  counterexamples in the range $q = O\lt({s^2}/{\log s}\rt)$. On the other hand,
  when $q$ is larger than some constant times ${s^2}/{\log s}$,
  we confirm that the Tur\'{a}n graph $T_s(n)$ asymptotically is the
  extremal graph achieving the maximum number of $q$-colorings. Furthermore,
  other (new and old) results on various instances of
  the Linial-Wilf problem are also established.

\end{abstract}


\input{intro}
\input{optimization}
\input{structure}
\input{applications}
\input{counterexample}
\input{integer}
\input{conclusion}

\input{appendix}

\end{document}

%% file: intro.tex
\section{Introduction}
\label{sec:intro}
A \emph{proper $q$-coloring} of a graph $G$ is an assignment mapping every
vertex to one of the $q$ colors in such a way that no two adjacent vertices
receive the same color. Let $P_G(q)$ denote the number of proper $q$-colorings
in a graph $G$. Introduced by Birkhoff \cite{Birk} in 1912, who proved that
$P_G(q)$ is always a polynomial in $q$, this important graph parameter,
as now commonly referred to as \emph{chromatic polynomial} of $G$,
has been extensively investigated over the past century.
As it is already NP-hard to determine whether this number $P_G(q)$ is nonzero
(even for $q=3$ and planar graph $G$), the focus of substantial research has
been to obtain good bounds for $P_G(q)$ over various families of graphs.

The original motive for Birkhoff~\cite{Birk} to consider the chromatic
polynomial was the famous four-color conjecture (now a theorem), which
equivalently asserts that the minimum $P_G(4)$ over all planar graphs
is at least one. For every $q\ge 5$, it was obtained by Birkhoff in
\cite{Birk30} that $P_G(q)\ge q(q-1)(q-2)(q-3)^{n-3}$
for every planar graph $G$ with $n$ vertices, which is also sharp.
Motivated from computational complexity, Linial~\cite{Lin} arrived at
the problem of minimizing the number of \emph{acyclic orientations} of
graph $G$, which equals $|P_G(-1)|$ by a result of Stanley~\cite{Stan},
over the family $\mc{F}_{n,m}$ of graphs with $n$ vertices and $m$ edges.
He gave a surprising answer that for any $n,m$, there exists a universal graph
minimizing $|P_G(q)|$ over the family $\mc{F}_{n,m}$ for \emph{every}
integer $q$. This graph is obtained from a clique $K_k$ by adding  $n-k-1$
isolated vertices and an extra vertex adjacent to $l$ vertices of the
clique $K_k$, where $k>l\ge 0$ are the unique integers satisfying
$\binom{k}{2}+l=m$.

Linial~\cite{Lin} then asked for the counterpart of his result, that is,
to maximize $|P_G(q)|$ over all graphs with $n$ vertices and $m$ edges
for integers $q$.
Wilf (see~\cite{BW,Wilf}) independently raised the same maximization problem
from a different point of view, the \emph{backtracking} algorithm for finding
a proper $q$-coloring.
Since then, this problem has been the subject of extensive research,
and many upper bounds on $P_G(q)$ over the family $\mc{F}_{n,m}$ have
been obtained (see, for instance, \cite{Byer,Laz89,Laz90,Liu,Nor}).
The case $q=2$ (for all $n,m$) was solved by Lazebnik in \cite{Laz89}
completely. In the same paper, Lazebnik conjectured that in the range
$m\le n^2/4$, the graphs with $n$ vertices and $m$ edges maximizing the
number of $3$-colorings must be complete bipartite graphs $K_{a,b}$ minus
the edges of some star, plus isolated vertices. This was confirmed in a
breakthrough paper \cite{LPS} of Loh, Pikhurko and Sudakov, who further
determined the asymptotic values of $a,b$. For $q\ge 4$, they also showed
that the same graphs achieve the maximum number of $q$-colorings, for all
sufficiently large $m<\kappa_q n^2$ where $\kappa_q \approx 1/(q\log q)$.
In fact, Loh \emph{et al.}~\cite{LPS} provided a general approach
which enables to find the asymptotic solution of the Linial-Wilf problem
by reducing it to a quadratically constrained linear problem, which we shall
introduce in \secref{sec:optimization}.
Despite the efforts by various researchers, still very little was known
for general $m,n,q$. ``Perhaps part of the difficulty for general $m,n,q$
stems from the fact that the maximal graphs are substantially more complicated
than the minimal graphs that Linial found'' (quoted from~\cite{LPS}).

The first contribution of our paper is a structural theorem that allows us
to substantially simplify the quadratically constrained linear problem for
general instances. Here we state it in a graph theoretic fashion and direct
readers to \thmref{thm:structure} for the specific statement linking to the
optimization problem. This structural theorem asserts that extremal graphs
must be asymptotically ``close'' to the ones in some family $\mc{G}_k$,
where $k>1$ is an integer only desponding on the edge density of graphs.
To be precise, for fixed $q$, the family $\mc{G}_k$ consists of
complete multipartite graphs with at least $k$ and at most $q$ parts as well as
graphs obtained from a complete $k$-partite graph by adding some additional
vertices each of which adjacent to the vertices of all but two fixed parts.
To measure the ``closeness", we define the \emph{edit distance} of two graphs
with the same number of vertices to be the minimum number of edges that need
to be added or deleted from one graph to make it isomorphic to the other.
We say $G$ is {\it $d$-close} to $H$, if the edit distance between graphs
$G$ and $H$ is at most $d$.

\begin{theorem}

  \label{thm:structure-graphs}
  For any real $s>1$, the following holds for all sufficiently
  large $n$. Let $G$ be an $n$-vertex graph with $\frac{s-1}{2s}n^2+o(n^2)$
  edges which maximizes the number of $q$-colorings over graphs with the
  same number of vertices and edges. Then there exists an $n$-vertex graph
  in $\mc{G}_{\cei{s}}$ which is $o(n^2)$-close to $G$.

\end{theorem}

\noindent We point out that the proofs of the following results will heavily
rely on the structure given by \thmref{thm:structure}, and we expect that
it will also provide useful insights to other unsolved ranges of the
Linial-Wilf problem.

Let $T_s(n)$ denote the complete $s$-partite $n$-vertex graph with
nearly-equal parts, i.e., the \emph{balanced Tur\'{a}n graph} with $n$
vertices and $s$ parts. Unlike the complicated situation in the general case,
Lazebnik conjectured (see~\cite{LPW}) that the Tur\'{a}n graphs $T_s(n)$ are
always extremal whenever $q\ge s$. More specifically,
\begin{conjecture}[Lazebnik]

  \label{conj:Lazebnik}
  For integers $q\ge s> 1$ and $n$ divisible by $s$, the Tur\'{a}n
  graph $T_s(n)$ has more proper $q$-colorings than any other graph with the
  same number of vertices and edges.

\end{conjecture}

\noindent The case when $q=s$ immediately follows from the well-known
Tur\'an's theorem. Lazebnik~\cite{Laz91} confirmed this when $q=\Omega(n^6)$,
and proved with Pikhurko and Woldar~\cite{LPW} that $T_2(2n)$ is extremal when
$q=3$ and is asymptotically extremal when $q=4$.
Loh, Pikhurko and Sudakov~\cite{LPS} showed that when $q=s+1\ge 4$,
$T_s(n)$ is the unique extremal graph for sufficiently large $n$,
which was improved to all $n$ by Lazebnik and Tofts in \cite{LazTofts}.
And Norine~\cite{Nor} partially confirmed this conjecture for sufficiently
large $n$, provided that $s$ divides $q$.
Very recently, Tofts~\cite{Tofts} proved that when $q=4$ and $s=2$, $T_2(n)$
is the unique extremal graph for all $n$. Surprisingly, despite these positive
results, we disprove \conjref{conj:Lazebnik} by the following.
\begin{theorem}

  \label{thm:counterexample}
  For all integers $s\ge 50000$ and $q_0$ such that $20s\le q_0 \le
  \frac{s^2}{200\log s}$, there exists an integer $q$, within
  distance at most $s$ from $q_0$, such that \conjref{conj:Lazebnik} is
  false for $(s,q)$.

\end{theorem}

On the other hand, we show that \conjref{conj:Lazebnik} asymptotically holds
for all large $s$ and $q\ge \frac{100s^2}{\log s}$.
\begin{theorem}

  \label{thm:main}
  For sufficiently large integers $s$ and $q$ such that
  $q\ge \frac{100s^2}{\log s}$, the following holds for all sufficiently large
  $n$. Every extremal graph which maximizes the number of proper $q$-colorings
  over graphs with $n$ vertices and $\frac{s-1}{2s}n^2+o(n^2)$
  edges is $o(n^2)$-close to the Tur\'an graph $T_s(n)$.

\end{theorem}

\noindent The above two results together show that for fixed integer $s$,
the order of magnitude $\frac{s^2}{\log s}$ is the threshold for the
number of colors $q$ (up to a constant factor): there are many
counterexamples to \conjref{conj:Lazebnik} when $q$ is smaller than
$\frac{s^2}{200\log s}$, while the Tur\'an graph $T_s(n)$  asymptotically
becomes optimal once $q$ exceeds $\frac{100s^2}{\log s}$.
We will discuss related issues about the precise structure of extremal graphs
in \secref{sec:conclusion}.

In the next result, we consider \conjref{conj:Lazebnik} for integers $s$,
where $q$ is not far from $s$.
\begin{theorem}

   \label{thm:q=s+C}
   (i). If $s\le q\le s+2$, then for all integers $s>1$ and sufficiently
   large integers $n$, every extremal graph which maximizes the number of
   $q$-colorings over graphs with $n$ vertices and $\frac{s-1}{2s}n^2+o(n^2)$
   edges is $o(n^2)$-close to the Tur\'{a}n graph $T_s(n)$.

   (ii). If $s+3\le q\le 2s-7$, where $s\ge 10$ is an integer, then
   \conjref{conj:Lazebnik} is false for $(s,q)$.

\end{theorem}

\subsection{Notation}

A graph $G=(V,E)$ is given by a pair of its (finite) vertex set $V(G)$ and
edge set $E(G)$, which does not contain loops or multi-edges. For a subset
$X$ of vertices, we use $G[X]$ to denote the subgraph of $G$ induced by $X$,
and write $e_G(X)$ to denote the number of edges in $G[X]$ (when it is clear
from the context, we will drop the subscript for brevity).
We will make the convention that the set of colors used in a proper
$q$-coloring is $[q]:=\{1, 2, \ldots, q\}$, and for the remainder of
the paper, we consider $q$ as a fixed integer parameter.
We also adopt the typographic convention for representing a
vector using boldface type, as in $\mathbf{v}$ for a vector whose
coordinates are $v_i$. To simplify the presentation, we often omit
floor and ceiling signs whenever these are not crucial and make no
attempts to optimize the absolute constants involved. All our
asymptotic notation symbols ($O$, $o$, $\Omega$, $\omega$, $\Theta$)
are relative to the number of vertices $n$, unless otherwise specified
with a subscript. Finally, the function $\log$ refers to the natural
logarithm.

\subsection{Organization}

In \secref{sec:optimization} we introduce the optimization
problem of \cite{LPS}, which asymptotically reduces the Linial-Wilf question to
a quadratically constrained linear problem, some tools used in mathematical
optimization, and some related results concerning the stability. We then study
the structure of solutions to the optimization problem for general instances
and finish the proof of \thmref{thm:structure-graphs} in \secref{sec:structure}.
In \secref{sec:applications}, we present some prompt applications of the
structural result obtained in the previous section, and we find the solutions for
certain instances including \thmref{thm:q=s+C} (i). The counterexamples
to \conjref{conj:Lazebnik} will be given in \secref{sec:counterexample},
which prove \thmref{thm:counterexample} and \thmref{thm:q=s+C} (ii).
In \secref{sec:opt_integer_s} we work on the optimization for integers $s$
by considering certain continuous relaxation, which leads to a complete proof
of \thmref{thm:main}. The final section contains some concluding remarks and
open problems.

%% file: optimization.tex
\section{The optimization problem}
\label{sec:optimization}
The problem of maximizing the number of proper $q$-colorings can be
asymptotically reduced to a quadratically constrained linear program,
as shown by Loh et al in \cite{LPS}. In this section, we describe this
optimization problem. Following the notation of \cite{LPS}, we define the
objective and constraint functions
\[
  \OBJ(\valpha) := \sum_{A\ne \emptyset} \alpha_A \log |A|, \quad
  \V(\valpha):=\sum_{A\ne\emptyset} \alpha_A, \Text{and}
  \E(\valpha):=\sum_{A\cap B = \emptyset} \alpha_A \alpha_B,
\]
where the vector $\valpha$ has $2^q-1$ coordinates $\alpha_A\in \R$ indexed by
nonempty subsets $A\subseteq [q]$ of colors, and the sum in $\E(\valpha)$
runs over all \emph{unordered pairs} of disjoint nonempty sets $\{A, B\}$.
We shall sometimes write $\sum_A$ in place of $\sum_{A\ne\emptyset}$, when it is
clear from the context that the empty set is excluded. Let $\FEAS(s)$
be the \emph{feasible set} of vectors defined by the constraints
$\valpha \ge 0$, $\V(\valpha)=1$, and $\E(\valpha)\ge \frac{s-1}{2s}$,
where $1< s \le q$ is a real parameter (not necessarily integer).

\medskip
  \underline{\OPTprob}.
  Determine $\OPT(s):=\max_{\valpha\in \FEAS(s)}\OBJ(\valpha)$.
\medskip

The maximum value $\OPT(s)$ exists since $\FEAS(s)$ is compact. We remark that
our notation is slightly different from the notation in \cite{LPS}, as we
replaced the parameter $\gamma$ used in \cite{LPS} (corresponding to the edge
density of the target graph) with the $\frac{s-1}{2s}$ instead. Our choice of
parameter was motivated by fact that the balanced Tur\'{a}n graph $T_s(n)$
has approximately $\frac{s-1}{2s}\cdot n^2$ edges, where $1< s \le q$
is an integer number.

We are ready to state the main theorem from \cite{LPS}, which asymptotically
reduces the original problem of maximizing $q$-colorings to the previous
optimization problem.

\begin{theorem}

  \label{thm:pob}
  For any $\eps>0$, the following holds for any sufficiently large $n$, and
  any $1< s \le q$. \\
  \hspace*{10pt}
    (i) Every $n$-vertex graph with at least $\frac{s-1}{2s} n^2$
        edges has fewer than $e^{(\OPT(s)+\eps)n}$ proper $q$-colorings. \\
  \hspace*{10pt}
    (ii) Any $\valpha$ which solves $\OPT(s)$ yields a graph $G_{\valpha}(n)$
         which has at least $\frac{s-1}{2s}n^2-2^q n$ edges and more than
         $e^{(\OPT(s)-\eps)n}$ proper $q$-colorings.

\end{theorem}

The construction of $G_\valpha(n)$ in \cite{LPS} is as follows.
Partition the vertex set of $G_{\valpha}(n)$ into clusters $V_A$ such
that $|V_A|$ differs from $\alpha_A n$ by less than $1$, and for every disjoint
pair $A, B\subseteq [q]$, join every vertex in $V_A$ to every vertex of $V_B$
by an edge. Assume that $\valpha$ solves $\OPT(s)$, it is easy to 
show that the number of proper $q$-colorings of
$G_{\valpha}(n)$ is roughly $e^{(\OPT(s)+o(1))n}$, and if $m$ denotes
the number of edges of $G_\valpha(n)$ then $\lt|m -\frac{s-1}{2s} n^2\rt|
\le 2^q n$.

To prove the stability of our results, we need to following statement
from \cite{LPS}.

\begin{theorem}

  \label{thm:stability-LPS}
  For any real $\eps>0$ and $s>1$, the following holds for all sufficiently
  large $n$. Let $G$ be an $n$-vertex graph with $m\le \frac{s-1}{2s}n^2$
  edges which maximizes the number of $q$-colorings. Then $G$ is
  $\eps n^2$-close to $G_\valpha(n)$, for an $\valpha$ which solves
  $\OPT(s')$ for some $|\frac{s'-1}{2s'}-\frac{m}{n^2}|<\eps$ with $s'\le s$.

\end{theorem}

Using \thmref{thm:stability-LPS} together with the continuity of $\OBJ$, $\V$,
$\E$ and $\OPT$ (for the continuity of $\OPT$ we refer the interested
reader to \cite[Claim 5, p.661]{LPS}), one can derive a more
convenient statement whose proof we defer to the appendix (see
\secref{sec:stability}).

\begin{corollary}

  \label{cor:stability-LPS}
  For any real $s>1$, the following holds for all sufficiently
  large $n$. Let $G$ be an $n$-vertex graph with $m=\frac{s-1}{2s}n^2 + o(n^2)$
  edges which maximizes the number of $q$-colorings. Then $G$ is
  $o(n^2)$-close to $G_\valpha(n)$ for some $\valpha$ which solves
  $\OPT(s)$.

\end{corollary}

\subsection{Some preliminaries about mathematical optimization}
\label{subsec:KKT}

\noindent One commonly used tool in mathematical optimization is the method of
\emph{Lagrange multipliers}, which is used to find the extrema of a multivariate
function $f(\vx)$ subject to the constraints $g_i(\vx)=0$ for $i=1, \ldots, m$,
where the \emph{objective function} $f:\R^n\to\R$ and
\emph{constraint functions} $g_i:\R^n\to \R$ are continuously differentiable.
This method asserts that if $\vx_0$ is a \emph{regular} local extremum for $f$,
then there exist constants $\lambda_i$ (the \emph{Lagrange multipliers}) such
that
\begin{equation}
  \label{eqn:lagrange}
  \nabla f(\vx_0) = \sum_{i=1}^m \lambda_i \nabla g_i(\vx_0),
\end{equation}
where a point $\vx_0\in\R^n$ is called \emph{regular} if the gradients $\nabla
g_1(\vx_0),\nabla g_2(\vx_0), \ldots, \nabla g_m(\vx_0)$ are linearly
independent over $\R$.

An extension of the method of Lagrange multipliers is the
\emph{method of Karush-Kuhn-Tucker} (see, e.g., \cite{BV}), or KKT for short.
Consider the following optimization problem:
\begin{equation}
  \label{eqn:kkt_optimization}
  \text{Maximize}\enskip f(\vx), \enskip \text{subject to}\enskip g_i(\vx)\le 0
  \enskip\text{and}\enskip h_j(\vx) = 0,
\end{equation}
where $f$ is the \emph{objective function}, $g_i$ ($i=1, 2, \ldots, m$) are the
\emph{inequality constraint functions} and $h_j$ ($j=1, 2, \ldots, l$) are the
\emph{equality constraint functions}. If $f:\R^n\to\R$, $g_i:\R^n\to \R$,
$h_j:\R^n\to \R$ are continuously differentiable at a point $\vx_0$, and $\vx_0$
is a local extremum for $f$ that satisfies some \emph{regularity conditions}
(see below), then there exist constants $\mu_i$ ($i=1, \ldots, m$)
and $\lambda_j$ ($j = 1, \ldots, l$), named \emph{KKT multipliers}, such
that the following hold
\begin{equation}
  \label{eqn:kkt_conditions}
  \begin{array}{rcl}
    \nabla f(\vx_0)  &=& {\displaystyle \sum_{i=1}^m \mu_i \nabla g_i(\vx_0) +
                         \sum_{j=1}^l \lambda_j \nabla h_j(\vx_0)}\\
    \mu_i g_i(\vx_0) &=& 0 \Text{for} i=1, 2, \ldots, m\\
    \mu_i            &\ge& 0 \Text{for} i=1, 2, \ldots, m.
  \end{array}
\end{equation}
A point $\vx_0$ is \emph{regular} for the KKT optimization if it satisfies some
constraint qualifications, such as:
\begin{itemize}
  \item Linear independence constraint qualification (LICQ): the gradients of
        the \emph{active} inequality constraints (the set of all constraints
        $g_i(\vx_0)\le 0$ for which the equality holds) and the gradients of the
        equality constraints are linearly independent at $\vx_0$;

  \item Slater's condition: If the equality constraints are given by linear
        functions $h_j$, and the inequality constraints are given by convex
        functions $g_i$, and there exists a point $\vx_1$ such that
        $h_j(\vx_1) = 0$ for all $j$ and $g_i(\vx_1) < 0$ for all $i$, then any
        point $\vx_0$ is regular.
\end{itemize}

\subsection{Basic properties of the solutions}
In this section we prove some basic results that will be used later
in the paper. We begin with a proposition asserting that $\E(\valpha)$
should be least possible for the optimal feasible $\valpha$. We remark that
the next statement appears implicitly in \cite{LPS}.

\begin{proposition}

  \label{prop:tight}
  Let $\valpha \in \FEAS(s)$ be a local maximum point for $\OBJ$. Then
  $\E(\valpha)=\frac{s-1}{2s}$.

\end{proposition}
\begin{proof}
Assume, towards contradiction, that there exists a local maximum point
$\valpha$ to $\OBJ$ such that $\E(\valpha) > \frac{s-1}{2s}$. Since
$\E(\valpha) > 0$, there exists a nonempty set $A\subseteq [q]$ such
that $\alpha_A \cdot \sum_{B\cap A = \emptyset} \alpha_B > 0$.
In particular $A \ne [q]$ and $\alpha_A > 0$.
Let $\valpha'$ be obtained from $\valpha$ by setting
$\alpha'_{A}:=\alpha_A - \eps$, $\alpha'_{[q]} := \alpha_{[q]} + \eps$
and keeping all the other entries unchanged, where $\eps > 0$ is small. We
have $\V(\valpha')=\V(\valpha)=1$,
\[
  \E(\valpha') = \E(\valpha) - \sum_{B\cap A =\emptyset} \eps\cdot\alpha_B
  \quad\Text{and}\quad
  \OBJ(\valpha') = \OBJ(\valpha) + \eps \cdot\log \frac{q}{|A|},
\]
and by taking $\eps$ sufficiently small, we can obtain
$\valpha'\in\FEAS(s)$ such that $\OBJ(\valpha')>\OBJ(\valpha)$, which
contradicts the local maximality of $\valpha$, concluding the proof of
this proposition.
\end{proof}

Since $\FEAS(s')$ is always a subset of $\FEAS(s)$ for $s'\ge s$,
we have
\begin{observation}

  \label{obs:monotone}
  $\OPT(s)$ is strictly decreasing with respect to the parameter $s$.

\end{observation}

We finish the section with two fairly straightforward propositions
about the sets in the support of a solution. The first one
is as follows.
\begin{proposition}

  \label{prop:color_cover}
  For any solution $\valpha$ of $\OPT(s)$, we have $\bigcup
  \{A: \alpha_A > 0\} =[q]$.

\end{proposition}
\begin{proof}
Assume, for contradiction, that $\bigcup\{A: \alpha_A > 0\}\ne [q]$
and let $B = [q] \setminus \bigcup\{A: \alpha_A > 0\}$. Let $A$
be any set in the support of $\valpha$. If we ``replace'' $A$ by
$A\cup B$, i.e., if we let $\valpha'$ be the vector obtained from
$\valpha$ such that $\alpha'_{A\cup B} = \alpha_A$, $\alpha'_A = 0$
and $\alpha'_X = \alpha_X$ for all $X\not\in \{A\cup B, A\}$, then
$\V(\valpha') = \V(\valpha)$, $\E(\valpha')=\E(\valpha)$
and $\OBJ(\valpha') > \OBJ(\valpha)$, a contradiction.
\end{proof}

The last proposition gives us some information about the support
whenever it contains an intersecting pair.

\begin{proposition}

  \label{prop:supp_inter}
  Let $\valpha$ be a solution to $\OPT(s)$, and suppose there exist
  two sets $A,B\subseteq [q]$ in the support of $\valpha$ such that $A\cap B
  \ne \emptyset$. Then either $A\subseteq B$ or for every $c\in A\setminus B$
  there exists a set $C$ in the support of $\valpha$ which is disjoint from $B$
  and contains $c$.

\end{proposition}
\begin{proof}
Suppose, towards contradiction, that there exists $c\in A\setminus B$ such that
no set $C$ in the support of $\valpha$ is disjoint from $B$ and contains $c$.
Let $\valpha'$ be the vector obtained from $\valpha$ by replacing
$B$ with $B\cup \{c\}$, i.e., $\alpha'_B := 0$, $\alpha'_{B\cup \{c\}} =
\alpha_{B\cup \{c\}} + \alpha_B$, and $\alpha'_X = \alpha_X$ for all
$X\not\in\{ B, B\cup \{c\}\}$. Clearly $\V(\alpha')=\V(\alpha) = 1$.
Moreover $\E(\valpha') = \E(\valpha)$, as every set in
the support of $\valpha$ that intersects $B\cup \{c\}$ must also intersect $B$.
But $\OBJ(\valpha') = \OBJ(\valpha) + \alpha_B \log\frac{|B|+1}{|B|}
> \OBJ(\valpha)$, a contradiction.
\end{proof}

%% file: structure.tex
\section{The structure of the support graph}
\label{sec:structure}
Let $\valpha\in\FEAS(s)$. The \emph{support graph} of $\valpha$, denoted by
$\SUPP(\valpha)$, is the graph defined over the support of $\valpha$ (that is,
those sets $A$ for which $\alpha_A>0$) whose edges are formed by
connecting pairs of disjoint sets. We will investigate the structure of
$\SUPP(\valpha)$ for all solutions $\valpha$ to $\OPT(s)$ in this section.

We define two classes of support graphs as follows.
Let $\mc{P}_k$ be the class of all $\SUPP(\valpha)$ for which
the support of $\valpha$ forms a $k$-partition of
$[q]$. Let $\mc{Q}_k$ be the class of all $\SUPP(\valpha)$ for
which the support of $\valpha$ consists of a $k$-partition
$A_1, \ldots, A_{k-1}, A_{k}$ of $[q]$ together with the set
$A_1\cup A_2$. We write $\mc{P}:=\cup_{k}\mc{P}_k$ and
$\mc{Q}:= \cup_{k}\mc{Q}_k$. In this section, we show
\begin{theorem}

  \label{thm:structure}
  For integer $q$ and real $1<s\le q$, all solutions $\valpha$ to
  $\OPT(s)$ are such that $\SUPP(\valpha)$ is in either
  $\cup_{\cei{s}\le k\le q}\mc{P}_k$ or $\mc{Q}_\cei{s}$. And when
  $\cei{s}<q$, we have $\SUPP(\valpha)\notin \mc{P}_q$.

\end{theorem}

\noindent This result tells us about the structure of the support graph for
a solution to $\OPT(s)$ and will play critical role in solving the relevant
instances of $\OPT(s)$. One can check that \thmref{thm:structure} together
with \corref{cor:stability-LPS} readily implies our main structural result
\thmref{thm:structure-graphs}. The proof of \thmref{thm:structure} will be
divided into seven steps, outlined below.

\medskip

\Step[Step 1] There exists a solution $\valpha$ to $\OPT(s)$ such
that none of the graphs $3K_1$, $C_4$ and $C_5$ appear as induced subgraphs of
$\SUPP(\valpha)$ (see \figref{fig:forb_graphs});

\begin{figure}[ht!]
\centering
\parbox{1.5in}{%
\centering
\begin{tikzpicture}
  [scale=0.8, auto=left, every node/.style={circle, draw, fill=black!50,
   inner sep=0pt, minimum width=4pt}]   \node (n1) at (90:1 cm) {};
  \node (n2) at (210:1 cm) {};
  \node (n3) at (-30:1 cm) {};
\end{tikzpicture}
\caption*{$3K_1$}
}
\parbox{1.5in}{%
\centering
\begin{tikzpicture}
  [scale=0.8, auto=left, every node/.style={circle, draw, fill=black!50,
   inner sep=0pt, minimum width=4pt}]
  \node (n1) at (45:1 cm) {};
  \node (n2) at (135:1 cm) {};
  \node (n3) at (-135:1 cm) {};
  \node (n4) at (-45:1 cm) {};

  \foreach \from/\to in {n1/n2, n2/n3, n3/n4, n4/n1}
    \draw (\from) -- (\to);
\end{tikzpicture}
\caption*{$C_4$}
}
\parbox{1.5in}{%
\centering
\begin{tikzpicture}
  [scale=0.8, auto=left, every node/.style={circle, draw, fill=black!50,
   inner sep=0pt, minimum width=4pt}]
  \node (n1) at (18:1 cm) {};
  \node (n2) at (90:1 cm) {};
  \node (n3) at (162:1 cm) {};
  \node (n4) at (234:1 cm) {};
  \node (n5) at (306:1 cm) {};

  \foreach \from/\to in {n1/n2, n2/n3, n3/n4, n4/n5, n5/n1}
    \draw (\from) -- (\to);
\end{tikzpicture}
\caption*{$C_5$}
}
\caption{Forbidden graphs.}
\label{fig:forb_graphs}
\end{figure}

\Step[Step 2] For any solution $\valpha$ to $\OPT(s)$, if $\SUPP(\valpha)$
does not contain an induced copy of $3K_1$, then it also does not contain
an induced matching of size $2$;

\medskip

\Step[Step 3] For any solution $\valpha$ to $\OPT(s)$, there exist
no four subsets $A, B, C$ and $D$ in the support of $\valpha$ such that
they induce a path $A-B-C-D$ in $\SUPP(\valpha)$ and $|A|>|C|$, $|D|>|B|$;

\medskip

\Step[Step 4] For any solution $\valpha$ to $\OPT(s)$, if $\SUPP(\valpha)$
can be obtained from a clique by removing a star (a collection of edges sharing
a common endpoint), then the star must have exactly two edges, and
$\SUPP(\valpha)\in\mc{Q}$;

\medskip

\Step[Step 5] For any solution $\valpha$ to $\OPT(s)$, if
$\SUPP(\valpha)\in\mc{Q}$ then in fact $\SUPP(\valpha)\in
\mc{Q}_\cei{s}$;

\medskip

\Step[Step 6] For any $\valpha$ from Step 1, it is true that
$\SUPP(\valpha)\in \lt(\cup_{k\ge\cei{s}}\mc{P}_k\rt)\cup
\mc{Q}_\cei{s}$.

\medskip

\Step[Step 7] All solutions to $\OPT(s)$ are such that $\SUPP(\valpha)$
is in either $\cup_{\cei{s}\le k\le q}\mc{P}_k$ or $\mc{Q}_\cei{s}$.

\medskip

\noindent We will show the proofs of the above steps in the forthcoming
subsections.

\subsection{Step 1: excluding graphs with 3 nonnegative eigenvalues}
We first prove the following proposition which shall be frequently used
throughout the section.

\begin{proposition}

  \label{prop:lin_dep}
  Let $\valpha$ be a solution to $\OPT(s)$ and let
  $A_1,A_2,\ldots,A_k \subseteq [q]$ be vertices in $\SUPP(\valpha)$.
  Let $\mathbf{1}, \mathbf{g},\vbeta\in \R^k$ be defined as $\textup{1}_i:=1$,
  $\beta_i :=\sum_{X\cap A_i=\emptyset} \alpha_X$,
  and $\textup{g}_i:=\log|A_i|$, for $i=1,2,\ldots,k$.
  Then the vectors $\mathbf{1}, \mathbf{g}, \vbeta$ are linearly
  dependent over $\R$.

\end{proposition}
\begin{proof}
Suppose, towards contradiction, that $\mathbf{1}$, $\mathbf{g}$, and
$\vbeta$ are linearly independent over $\R$. Then there exists a vector
$\vgamma\in \R^k$ such that $\mathbf{1}\cdot \vgamma = 0$, while
$\mathbf{g}\cdot \vgamma = \vbeta\cdot \vgamma = 1$, where $(\cdot)$
denotes the standard inner product in $\R^k$. Let $\valpha'$ be
obtained from $\valpha$ by replacing $\alpha'_{A_i} :=
\alpha_{A_i} + \eps \cdot \gamma_i$ for $i=1,\ldots, k$,
where $\eps>0$ is sufficiently small. Clearly $\valpha' \ge 0$ when
$\eps$ is sufficiently small and $\V(\valpha') = 1$, as
$\mathbf{1}\cdot \vgamma = 0$. We also have
\[
  \E(\valpha') = \E(\valpha) + \eps + O(\eps^2)
  \Text{and}\OBJ(\valpha') = \OBJ(\valpha) + \eps,
\]
therefore when $\eps$ is sufficiently small, $\valpha'\in \FEAS(s)$
and $\OBJ(\valpha') > \OBJ(\valpha)$, a contradiction that establishes
the proposition.
\end{proof}

We say $\lambda$ is an \emph{eigenvalue} of a graph $G$ when $\lambda$
is an eigenvalue of its adjacency matrix.
The following proposition is the main ingredient in the proof of Step 1.

\begin{proposition}

  \label{prop:3_eigenvalues}
  Let $\valpha$ be a solution to $\OPT(s)$, let
  $A_1,A_2,\ldots,A_k \subseteq [q]$ be vertices in $\SUPP(\valpha)$,
  and let $H$ be the subgraph of $\SUPP(\valpha)$ induced by
  these vertices. Then $H$ has at most two positive eigenvalues.
  Moreover, if $H$ has at least three nonnegative eigenvalues, then
  there exists another solution $\valpha'$ to $\OPT(s)$ such that
  $\SUPP(\valpha')$ is a strictly smaller induced subgraph of
  $\SUPP(\valpha)$. Furthermore $\valpha$ and $\valpha'$ differ only
  in the coordinates $A_1,A_2,\ldots, A_k$ and the segment that
  joins $\valpha$ and $\valpha'$ is entirely contained in $\FEAS(s)$.

\end{proposition}
\begin{proof}
Let $\mathbf{1}$, $\mathbf{g}$, and $\vbeta$ be as such as
in \propref{prop:lin_dep} for the vertices $A_1,\ldots, A_k$.
Because $\mathbf{1}$, $\mathbf{g}$, and $\vbeta$ are linearly dependent,
we have $\dim(\textup{span}\{\mathbf{1}, \mathbf{g}, \vbeta\})\le 2$.
Let $M$ denote the $k\times k$ adjacency matrix of $H$.
Let $W\subseteq \R^k$ be the subspace spanned by the eigenvectors
of $M$ associated with nonnegative eigenvalues. Because $M$ is symmetric,
if $H$ has at least 3 nonnegative eigenvalues, then $\dim(W)\ge 3$, so
there exists a vector $\vgamma\in W$ perpendicular to the vectors
$\mathbf{1}$, $\mathbf{g}$, and $\vbeta$. Since $\vgamma\in W$, we must
have $\vgamma^T \cdot M \cdot \vgamma \ge 0$. Let $\valpha'$ be
obtained from $\valpha$ by replacing $\alpha'_{A_i} :=
\alpha_{A_i} + \eps \cdot \gamma_i$, where $\eps>0$
will be chosen later. If $\eps$ is sufficiently small, we
have $\valpha'\ge 0$. Moreover, $V(\valpha') = 1$, and
\[
  \E(\valpha') = \E(\valpha) + \eps^2 \vgamma^T\cdot M \cdot \vgamma
  \ge \E(\valpha)  \Text{and}\OBJ(\valpha') = \OBJ(\valpha).
\]
Thus we must conclude that $\vgamma^T\cdot M \cdot \vgamma = 0$, otherwise
we would get a solution $\valpha'$ to $\OPT(s)$ with $\E(\valpha')>
\frac{s-1}{2s}$, contradicting \propref{prop:tight}. In particular,
$H$ does not have 3 positive eigenvalues (we could just repeat the same
argument replacing $W$ with the subspace spanned by the eigenvectors of $M$
associated with the 3 positive eigenvalues).
Clearly the segment joining $\valpha$ and $\valpha'$ is
entirely contained in $\FEAS(s)$, and by choosing $\eps$ appropriately,
we can make one extra coordinate of $\valpha'$ to be zero, thereby reducing
the size of the support of the solution, and concluding the proof.
\end{proof}

From the proof of \propref{prop:3_eigenvalues}, we remark the following.
\begin{observation}

  \label{obs:obs_kernel}
  Let $H$, $\mathbf{1}$, $\mathbf{g}$, and $\vbeta$ be such as in 
  \propref{prop:3_eigenvalues} and its proof.
  If $H$ has at least three nonnegative eigenvalues,
  then there exists a vector $\vgamma$ in the kernel of the adjacency
  matrix of $H$ such that $\vgamma$ is perpendicular to the vectors
  $\mathbf{1}$, $\mathbf{g}$, and $\vbeta$.

\end{observation}

Let $C_5^+$ be the 5-vertex graph obtained from $C_5$ by adding an edge.
It is easy to verify that the eigenvalues of $3K_1$ are $0$,$0$,
and $0$; the eigenvalues of $C_4$ are $2$, $0$, $0$,
and $-2$; the eigenvalues of $C_5$ are $2$, $2\cos\frac{2\pi}{5}$,
$2\cos\frac{2\pi}{5}$, $2\cos\frac{4\pi}{5}$, and $2\cos\frac{4\pi}{5}$;
and the eigenvalues of $C_5^+$ are $\lambda_1,\lambda_2,0,\lambda_3$, and $-2$,
where $\lambda_1\approx 2.48$, $\lambda_2\approx 0.69$,
$\lambda_3\approx -1.17$ are the roots to the equation
$\lambda^3-2\lambda^2-2\lambda+2=0$.
We remark that these graphs have three nonnegative eigenvalues each,
and that $C_5^+$ contains $C_4$ as an induced subgraph.
The following proposition finishes the proof of Step 1.

\begin{proposition}

  \label{prop:step_1}
  No solution $\valpha$ to $\OPT(s)$ has $C_5$ as induced subgraph
  of $\SUPP(\valpha)$. Furthermore, there exists a solution $\valpha$ to
  $\OPT(s)$ such that $\SUPP(\valpha)$ does not contain induced
  copies of $3K_1$, $C_5^{+}$ and $C_4$.

\end{proposition}
\begin{proof}
Let $\valpha$ be an arbitrary solution to $\OPT(s)$.
\propref{prop:3_eigenvalues} asserts that $\SUPP(\valpha)$
does not contain induced copies of $C_5$, since $C_5$ has three positive
eigenvalues.

If $\SUPP(\valpha)$ contains either $3K_1$, $C_5^+$ or $C_4$ as induced
subgraphs, we just repeatedly apply \propref{prop:3_eigenvalues} (with
$H=3K_1, C_5^+$ or $C_4$) to find a new solution $\valpha'$ with strictly
smaller support, until there are no more induced copies of these graphs in
$\SUPP(\valpha')$.
This is possible because each of $3K_1, C_5^+, C_4$ has three nonnegative
eigenvalues. It is important to remark here that
\begin{enumerate}[(i)]
\item whenever there are copies of $3K_1$ appearing as induced subgraphs of
$\SUPP(\valpha)$, we always choose to apply
\propref{prop:3_eigenvalues} to remove induced copies of $3K_1$
first, and
\item when $3K_1$ do not appear but there is a copy of $C_5^+$
in $\SUPP(\valpha)$, we apply \propref{prop:3_eigenvalues} to remove
the induced copies of $C_5^+$ (rather than removing induced copies of $C_4$);
\item finally, when there are no induced copies of $3K_1$ or $C_5^+$,
we apply \propref{prop:3_eigenvalues} to remove the remaining induced
copies of $C_4$.
\end{enumerate}
This priority (namely $3K_1> C_5^+> C_4$) will play an important role in the
proof of Step 7.

The process described above has to end eventually, because the
support always reduces in size at each application of
\propref{prop:3_eigenvalues}. At the end, we obtain a
solution $\valpha'$ to $\OPT(s)$ such that $\SUPP(\valpha')$ has no
induced copies of $3K_1$, $C_5^+$ and $C_4$, and $\valpha$ and $\valpha'$
are connected in $\FEAS(s)$ by a piecewise linear path.
\end{proof}

At last, we state a useful observation that will be needed in the
subsequent steps.

\begin{observation}

  \label{obs:two_color_cover}
  For any $\valpha$, if $\SUPP(\valpha)$ does not contain an induced
  copy of $3K_1$ then any color $i\in [q]$ is contained in at most
  two subsets of the support of $\valpha$.

\end{observation}

\subsection{Step 2: excluding matchings of size two}
\begin{proposition}

  \label{prop:step_2}
  Let $\valpha$ be a solution of $\OPT(s)$ such that $\SUPP(\valpha)$
  does not contain induced copies of $3K_1$. Then $\SUPP(\valpha)$
  does not contain induced matchings of size two.

\end{proposition}
\begin{proof}
Suppose for a contradiction that some $\valpha$ violates the
proposition. We know there exist four vertices $A,B,C,D$ of
$\SUPP(\valpha)$ such that $A\cap C=B\cap D=\emptyset$ and all other
intersections $A\cap B, B\cap C, C\cap D, D\cap A$ are nonempty.

Without losing generality, we assume that the intersection $A\cap B$ is of the
smallest size among all pairwise nonempty intersections of the sets in
$\{A,B,C,D\}$. There exists a subset $S$ of $C\cap D$ such that $|S|=|A\cap B|$.
Clearly $A\cap S = B \cap S = \emptyset$.
We define $B'=(B\setminus(A\cap B))\cup S$, and $D'=(D\setminus S)\cup
(A\cap B)$. By \obsref{obs:two_color_cover}, if $X$ is in the support of
$\valpha$ and $X\cap (A\cap B) \ne\emptyset$ then either $X=A$ or $X=B$.
Similarly, if $X\cap S \ne \emptyset$, then either $X = C$ or $X = D$.
In particular, $\alpha_{B'} = \alpha_{D'} = 0$. Let $\valpha'$ be a vector
obtained from $\valpha$ by defining $\alpha'_{B'}=\alpha_{B}$,
$\alpha'_{D'}=\alpha_{D}$, $\alpha'_B=0$, $\alpha'_D=0$,
and letting $\alpha'_{X} = \alpha_X$ for all $X \not\in\{B, B', D, D'\}$.
It is easy to see that $|B'| = |B|$ and $|D'|=|D|$,
therefore $\OBJ(\valpha')=\OBJ(\valpha)=\OPT(s)$.

The edges between $B',D'$ and another subset
$X \not\in \{A,B,B',C,D,D'\}$ in $\SUPP(\valpha')$ are the same as the
edges between $B,D$ and $X$ in $\SUPP(\valpha)$, as this
color-swapping does not affect those adjacencies.
We also have $ B'\cap D'=\emptyset$, while $A\cap B'= \emptyset \ne A \cap B$.
Therefore, $\E(\valpha')\ge \E(\valpha)+\alpha'_{A}\alpha'_{B'}>\E(\valpha)$.
Note that $\OBJ(\valpha')=\OBJ(\valpha)=\OPT(s)$, so $\valpha'$ and
$\valpha$ are both global maximum points for $\OBJ$ and by
\propref{prop:tight}, $\E(\valpha')=\E(\valpha)=\frac{s-1}{s}$, a
contradiction. This completes the proof of this proposition, thereby
proving Step 2.
\end{proof}

\subsection{Step 3: excluding paths with four vertices}
\begin{proposition}

  \label{prop:step_3}
  For any solution $\valpha$ to $\OPT(s)$, there exist
  no four subsets $A, B, C$ and $D$ in the support of $\valpha$ such that
  they induce a path $A-B-C-D$ in $\SUPP(\valpha)$ and $|A|>|C|$, $|D|>|B|$.

\end{proposition}
\begin{proof}
Suppose that for some solution $\valpha$ to $\OPT(s)$ there exist four subsets
$A$, $B$, $C$ and $D$ in the support of $\valpha$ such that they induce a path
$A-B-C-D$ in $\SUPP(\valpha)$ and $|A|>|C|$, $|D|>|B|$. By symmetry, we assume
that $|D|\ge |A|$. Let
\[
  M=\lt[\begin{array}{llll}
       0&0&1&1 \\
       0&0&0&1 \\
       1&0&0&0 \\
       1&1&0&0
    \end{array}\rt]
\]
be the adjacency matrix of the induced path $A-B-C-D$, where its rows/columns
are arranged according to the order $B$, $D$, $A$, $C$. Let $\lambda=
\frac{\sqrt{5}+1}{2}$. It can be verified that $M$ has four eigenvalues
$\lambda, -\lambda, \lambda-1, -\lambda+1$ and their corresponding eigenvectors
\[
  \mathbf{v}_1=\lt[
    \begin{array}{r} \lambda\\1\\1\\ \lambda\end{array}\rt],
  \mathbf{v}_2=\lt[
    \begin{array}{r} -\lambda\\-1\\ 1\\ \lambda\end{array}\rt],
  \mathbf{v}_3=\lt[
    \begin{array}{r} 1\\-\lambda\\ \lambda\\-1\end{array}\rt],
  \mathbf{v}_4=\lt[
    \begin{array}{r} 1\\-\lambda\\ -\lambda\\ 1\end{array}\rt].
\]
Note that the above four eigenvectors are orthogonal.
Let $\mathbf{1}$, $\mathbf{g}$, and $\vbeta$ be such as
in \propref{prop:lin_dep} for the vertices $B$,$D$,$A$,$C$.
We have $\mathbf{g} = (x_1, x_2, x_3, x_4)^T$, where $x_1=\log|B|$,
$x_2=\log|D|$, $x_3=\log|A|$ and $x_4=\log|C|$. We remark that $x_2>x_1$,
$x_3>x_4$ and $x_2\ge x_3$.

We claim that there exists a vector $\mathbf{v}:=a\mathbf{v}_1+b\mathbf{v}_2+
c\mathbf{v}_3+d\mathbf{v}_4$ such that $\mathbf{v}\cdot \mathbf{1}=0$,
$\mathbf{v}\cdot \mathbf{g}=0$ and $\mathbf{v}^T M \mathbf{v}=
2(\lambda^2+1)\cdot\lt[\lambda(a^2-b^2)+(\lambda-1)(c^2-d^2)\rt]>0$.
We first rewrite $\mathbf{v}\cdot\mathbf{1}=0$ as $a(1+\lambda)+d(1-\lambda)=0$,
which together with $\lambda^2=\lambda+1$ implies that $d=\frac{\lambda+1}%
{\lambda-1}\cdot a=\lambda^3 a$. Substituting $d=\lambda^3 a$, the second
equation $\mathbf{v}\cdot \mathbf{g}=0$ becomes
\[
  a(1+3\lambda)(x_1-x_2-x_3+x_4)+b(-\lambda x_1-x_2+x_3+\lambda x_4)+
  c(x_1-\lambda x_2+\lambda x_3-x_4)=0.
\]
If $x_1-\lambda x_2+\lambda x_3-x_4=0$, then we may choose $a$,$b$ such that
$a(1+3\lambda)(x_1-x_2-x_3+x_4)+b(-\lambda x_1-x_2+x_3+\lambda x_4)=0$ and
choose $c$ sufficiently large such that $\mathbf{v}^T M \mathbf{v}>0$, thereby
the claim follows. If $x_1-\lambda x_2+\lambda x_3-x_4\ne 0$, then we choose
$a=b=1$, $d=\lambda^3$ and
\[
  c=\frac{(3\lambda+2)x_2+3\lambda x_3-(2\lambda+1)x_1-(4\lambda+1)x_4}{%
  x_1-\lambda x_2+\lambda x_3-x_4}
\]
such that the second equation $\mathbf{v}\cdot\mathbf{g}=0$ is satisfied.
We will show $|c|>\lambda^3$, which implies $\mathbf{v}^T M\mathbf{v}>0$
and hence the claim. Note that $x_2>x_1, x_3>x_4$ and
$x_2\ge x_3$, so $(3\lambda+2)x_2+3\lambda x_3-(2\lambda+1)x_1-
(4\lambda+1)x_4>0$, thus it suffices to show that
\begin{align*}
  (3\lambda+2)x_2+3\lambda x_3-(2\lambda+1)x_1-(4\lambda+1)x_4&>
  \lambda^3(x_1-\lambda x_2+\lambda x_3-x_4)\Text{and}\\
  (3\lambda+2)x_2+3\lambda x_3-(2\lambda+1)x_1-(4\lambda+1)x_4&>
  -\lambda^3(x_1-\lambda x_2+\lambda x_3-x_4).
\end{align*}
To see this, using $\lambda^3=2\lambda+1$ and $\lambda^4=3\lambda+2$, the above
two inequalities can be simplified as $(6\lambda+4)x_2>2x_3+(4\lambda+2)x_1+
2\lambda x_4$ and $x_3>x_4$, which are obviously true, thereby
establishing the claim.

Rewrite this vector $\mathbf{v}$ as $(v_1,v_2,v_3,v_4)^T$. 
As in the proof of \propref{prop:3_eigenvalues}, we let $\valpha'$ be
a vector obtained from $\valpha$ by replacing
\begin{align*}
  \alpha'_{B} := \alpha_{B} + \eps \cdot v_1 &,\quad
  \alpha'_{D} := \alpha_{D} + \eps \cdot v_2\\
  \alpha'_{A} := \alpha_{A} + \eps \cdot v_3 &,\quad
  \alpha'_{C} := \alpha_{C} + \eps \cdot v_4
\end{align*}
while keeping the other entries unchanged, where $\eps>0$ is sufficiently small. We have
$\valpha'\ge 0$, $\V(\valpha') = 1$, and $\OBJ(\valpha')=\OBJ(\valpha)$;
because $\mathbf{1}$, $\mathbf{g}$ and $\vbeta$ are linearly dependent
(this implies $\mathbf{v}\cdot \vbeta=0$), we also have
\[
  \E(\valpha') = \E(\valpha) + \eps^2 \mathbf{v}^T M \mathbf{v}
  > \E(\valpha),
\]
which contradicts \propref{prop:tight}. This finishes the proof.
\end{proof}

\subsection{Step 4: the star has two petals}
\begin{proposition}

  \label{prop:step_4}
  For any solution $\valpha$ to $\OPT(s)$, if $\SUPP(\valpha)$ can be
  obtained from a clique by removing a star (a collection of edges sharing
  a common endpoint), then the star must have exactly two edges and
  $\SUPP(\valpha)\in\mc{Q}$.

\end{proposition}
\begin{proof}
Let $\valpha$ be a solution to $\OPT(s)$ such that $\SUPP(\valpha)$ is a graph
obtained from a clique by removing the edges of a star. Assume that the support
of $\valpha$ consists of sets $B_1,\ldots, B_t,A_1,\ldots,A_k,C$ for some $t\ge 0$ and
$k\ge 1$ such that sets $B_1, \ldots, B_t, A_1, \ldots, A_k$ are disjoint,
$C\cap B_i=\emptyset$ and $C\cap A_j\ne \emptyset$ for all $i,j$. Since
$\valpha$ is a solution to $\OPT(s)$, \propref{prop:supp_inter} implies that
every color in $C\cup A_1\cup\ldots\cup A_k$ is covered at least twice by
this union. This is only possible when $C=A_1\cup A_2\ldots \cup A_k$,
and thus $k\ge 2$. Moreover, \propref{prop:color_cover} implies that
$B_1,\ldots,B_t,A_1, \ldots,A_k$ form a partition of $[q]$.
To prove that $\SUPP(\valpha)\in \mc{Q}$, we need
to show that $k=2$. Assume, towards contradiction, that $k\ge 3$.

Without loss of generality, we assume that $C=[p]$ for some integer $p\le q$.
Instead of working on $[q]$ and its $\OPT(s)$, we will turn to the study
of the smaller ground set $[p]$ and a new optimization which can be
viewed as a restriction of $\OPTnone$ on $[p]$. Let $\gamma =\sum_{i=1}^k
\alpha_{A_i}+\alpha_C$ and $\widetilde{\valpha}$ be the vector with
$2^p-1$ coordinates such that its support consists of $A_1,\ldots,A_k,C$
and $\widetilde{\alpha}_{A_i}=\frac{\alpha_{A_i}}{\gamma}$,
$\widetilde{\alpha}_{C}=\frac{\alpha_{C}}{\gamma}$. Choose real $s'$ so that
$\frac{s'-1}{2s'}=\Eq{p}(\widetilde{\valpha})$. Now we consider the
new optimization problem $\OPTq{p}(s')$ restricted to the ground
set $[p]$. Since $\valpha$ solves $\OPT(s)$, by our definitions, it is clear
that $\widetilde{\valpha}$ solves $\OPTq{p}(s')$ as well.

We have seen that $\widetilde{\valpha}$ maximizes $\OBJq{p}(
\widetilde{\valpha})$ subject to $\alpha_A\ge 0$ for all nonempty
$A\subseteq [p]$, $\Eq{p}(\valpha)-\frac{s'-1}{2s'}\ge 0$ and
$\Vq{p}(\valpha)=1$.
We will apply the method of Karush-Kuhn-Tucker to $\widetilde{\valpha}$ and
$\OBJq{p}$ (recall the \secref{subsec:KKT}). Before proceeding, we
point out that $\widetilde{\valpha}$ is a regular point for $\OBJq{p}$,
as one can easily verify that $\widetilde{\valpha}$ satisfies the LICQ
conditions. For convenience, write $\widetilde{\alpha}_i:=
\widetilde{\alpha}_{A_i}$, $T:=\sum_{i=1}^k \widetilde{\alpha}_i$ and
$P:=\prod_{i=1}^k |A_i|$. Let $\pi_A$ denote the projection on the coordinate
$A$, i.e., $\pi_A(\valpha)=\alpha_A$. By \eqref{eqn:kkt_conditions},
there exist constants $\mu_A\le 0, \mu$ and $\lambda\le 0$ such that
\[
  \nabla \OBJq{p}(\widetilde{\valpha})=\lt(\sum_{A\subseteq [p]}\mu_A\cdot\nabla
  \pi_A\rt)+ \mu\cdot\nabla \Vq{p}(\widetilde{\valpha})+\lambda\cdot\nabla
  \Eq{p}(\widetilde{\valpha}).
\]
For subset $A$ with $\widetilde{\alpha}_A=0$, we have
\begin{equation}
  \label{eqn:step4_1}
  \log|A|=\mu_A+\mu+\lambda\cdot \sum_{B\cap A=\emptyset}\widetilde{\alpha}_B
  \le \mu+\lambda\cdot \sum_{B\cap A=\emptyset}\widetilde{\alpha}_B.
\end{equation}
For $C$ and $A_i$, we have $\mu_C=\mu_{A_i}=0$, so
\begin{equation}
  \label{eqn:step4_2}
  \log|C|=\mu \enskip\text{and}\enskip\log|A_i|=
  \mu+\lambda\cdot(T-\widetilde{\alpha}_i).
\end{equation}
By \eqref{eqn:step4_2} and summing $\log|A_i|$ from $i=1$ to $k$, we get that
$\lambda\cdot T=\frac{1}{k-1}\log\frac{P}{|C|^k}$ and thereby
\[
  \lambda\cdot \widetilde{\alpha}_i=\frac{1}{k-1}\log\lt(
  \frac{P}{|A_i|^{k-1}|C|}\rt).
\]
As $k\ge 3$, by substituting $A:=C\setminus A_i$ in \eqref{eqn:step4_1},
we obtain $\log(|C|-|A_i|)\le \log|C|+\lambda\cdot\widetilde{\alpha}_i$, thus
\[
  |A_i|^{k-1}\cdot(|C|-|A_i|)^{k-1}\le P\cdot|C|^{k-2}
\]
for all $i=1,2,\ldots,k$. Let $x_i:=\frac{|A_i|}{|C|}$ and assume $x:=x_1\ge
x_2\ge \ldots\ge x_k>0$, then $1>x\ge \frac{1}{k}$ as $\sum_{i=1}^k x_i=1$.
In addition, the previous inequality implies that
\begin{equation}
  \label{eqn:step4_4}
  x^{k-1}(1-x)^{k-1}\le x\cdot\prod_{j=2}^k x_j.
\end{equation}
For $k\ge 3$, one can check that $x^{k-2}\ge(\frac{1}{k})^{k-2}>(\frac{1}{k-1})^{k-1}$.
This inequality, toghether with the AM-GM inequality, implies
\[
  x^{k-1}(1-x)^{k-1}>x\cdot\frac{(1-x)^{k-1}}{(k-1)^{k-1}}\ge
  x\cdot\prod_{j=2}^k x_j,
\]
contradicting \eqref{eqn:step4_4} and finishing the proof of Step 4.
\end{proof}

\subsection{Step 5: computing the size of the support when
the solution is in $\mc{Q}$}
We introduce restricted versions of the \OPTprob.
Fix an integer $k\ge 1$ and a collection $\{A_1,\ldots,A_{k-1},
A_k, A_1\cup A_2\}$ of $[q]$ such that $A_1,\ldots,A_{k}$ form a
$k$-partition of $[q]$. We consider the following optimization:
\begin{equation}
  \label{opt:fixed_Qk}
  \begin{array}{ll}
    \text{Maximize}   & \sum_{i=1}^k\alpha_i \log|A_i|+
                        \beta\log(|A_1|+|A_2|)\\
    \text{Subject to} & \sum_{i=1}^k\alpha_i+\beta=1,\\
                      & \sum_{i=1}^k\alpha_i^2+\beta^2+
                        2\beta(\alpha_1+\alpha_2)\le \frac{1}{s},\\
                      &\alpha_i\ge 0, \beta\ge 0 \Text{for}i=1,\ldots, k.
  \end{array}
\end{equation}
The conditions of \eqref{opt:fixed_Qk} are consistent with the conditions of
the \OPTprob, when restricted to vectors with support in $\mc{P}_k\cup
\mc{Q}_k$. By compactness, the maximum of \eqref{opt:fixed_Qk} exists,
which is achieved either in the boundary or interior of its domain,
where the vectors in the interior have strict positive coordinates.
By Cauchy-Schwarz, we have
\begin{equation}
  \label{eqn:Qk_bound_k}
  \frac{1}{s}\ge \beta^2+2\beta(\alpha_1+\alpha_2)+\sum_{i=1}^k\alpha_i^2
  \ge (\beta + \alpha_1)^2 + \sum_{i=2}^k \alpha_i^2 \ge
  \frac{1}{k}\lt(\beta + \sum_{i=1}^k\alpha_i\rt)^2 = \frac{1}{k},
\end{equation}
which implies $k \ge s$, and since $k$ is an integer we must also have
$k\ge \cei{s}$. We also remark that if $\alpha_1+\alpha_2,\beta > 0$ then
the inequality is strict, i.e. $k>s$.
We derive a necessary condition for the maximum to be
attained in the interior of the domain of \eqref{opt:fixed_Qk}. Let
\begin{align*}
  S_1&:= \frac{1}{k-1}
  \lt[\log(|A_1| + |A_2|)+\sum_{i=3}^{k}\log|A_i|\rt],\Text{and} \\
  S_2&:=\frac{1}{k-1}\lt[-\log^2\lt(\frac{|A_1|\cdot|A_2|}{|A_1|+|A_2|}\rt)
   +\sum_{i=1}^{k} \log^2|A_i|\rt].
\end{align*}
\begin{lemma}

  \label{lem:fixed_Qk}
  If $\valpha=(\alpha_1,\ldots,\alpha_{k},\beta)>0$ is a local
  maximum point for \eqref{opt:fixed_Qk}, then
  \[
    \alpha_1=\frac{1}{\lambda}\log\lt(
      \frac{|A_1|+|A_2|}{|A_2|}\rt),\enskip
    \alpha_2=\frac{1}{\lambda}\log\lt(
      \frac{|A_1|+|A_2|}{|A_1|}\rt),\enskip
    \beta=\frac{1}{\lambda}\log\lt(\frac{|A_1|\cdot|A_2|}{%
      |A_1|+|A_2|}\rt)-\frac{\mu}{\lambda},
  \]
  for each $i=3,\ldots, k$, we have $\alpha_i=\frac{\log|A_i|-\mu}{\lambda}$,
  and the value of the objective function of \eqref{opt:fixed_Qk} is
  \[
    \beta\log(|A_1|+|A_2|)+\sum_{i=1}^{k}\alpha_i \log|A_i|=
    \mu+\frac{\lambda}{s},
  \]
  where $\lambda=(k-1)\sqrt{\frac{s}{k-1-s}\lt(S_2-S_1^2\rt)}>0$ and
  $\mu=S_1-\frac{\lambda}{k-1}$. In particular, we have $k\ne s+1$,
  $\log\lt(\frac{|A_1|\cdot|A_2|}{|A_1|+|A_2|}\rt)>\mu$, and
  $\log |A_i|> \mu$ for each $i\ge 3$.

\end{lemma}
\begin{proof}
By the same proof of \propref{prop:tight}, the local maximality of
\eqref{opt:fixed_Qk} implies $\sum_{i=1}^{k}\alpha_i^2+\beta^2+2\beta
(\alpha_1 + \alpha_2)=\frac{1}{s}$.
We can apply the KKT method as a straightforward calculation shows
that the local extremum $\valpha$ satisfies the LICQ conditions.
Therefore there exist constants $\lambda\ge 0$ and $\mu$ such
that for each $i \ge 3$, $\log|A_i|=\mu+\lambda\alpha_i$,
$\log|A_1|=\mu+\lambda(\alpha_1+\beta)$,
$\log|A_2|=\mu+\lambda(\alpha_2+\beta)$ and $\log(|A_1|+|A_2|)=\mu+
\lambda(\alpha_1+\alpha_2+\beta)$. If $\lambda=0$, then
$\log(|A_1|+|A_2|)=\mu=\log|A_1|$, a contradiction, thus $\lambda>0$.

We can rewrite the above equations as $\alpha_i=
\frac{1}{\lambda}(\log|A_i|-\mu)$ for each $i\ge 3$, $\alpha_1=
\frac{1}{\lambda}\log\lt(\frac{|A_1|+|A_2|}{|A_2|}\rt)$,
$\alpha_2=\frac{1}{\lambda}\log \lt(\frac{|A_1|+|A_2|}{|A_1|}\rt)$
and $\beta=\frac{1}{\lambda}\lt(\log\lt(
\frac{|A_1|\cdot|A_2|}{|A_1|+|A_2|}\rt)-\mu\rt)$. Now solving the
system $\sum_{i=1}^{k}\alpha_i+\beta=1$ and $(\alpha_1+\beta)^2+
(\alpha_2+\beta)^2 -\beta^2+\sum_{i=3}^{k}\alpha_i^2=\frac{1}{s}$, we obtain
the desired expressions of $\lambda$ and $\mu$. This finishes the proof.
\end{proof}

\medskip

For completion, we turn to the case $\beta = 0$. Let
$S_1':=\frac{1}{k}\sum_{i=1}^k \log |A_i|$ and $S_2':=\frac{1}{k}
\sum_{i=1}^k \log^2|A_i|$. We have the following

\begin{lemma}

  \label{lem:fixed_Pk}
  If $k>s$ and $\valpha=(\alpha_1,\ldots,\alpha_k,0)$ is a local maximum point
  of \eqref{opt:fixed_Qk} satisfying $\alpha_i > 0$ for all $i=1,\ldots, k$,
  then $\lambda'\alpha_i=\log|A_i|-\mu'$ for each $i=1,\ldots,k$ and the
  value of the objective function is
  \[
    \sum_{i=1}^k \alpha_i \log |A_i|=\mu'+\frac{\lambda'}{s},
  \]
  where $\lambda'=k\sqrt{\frac{s}{k-s}\lt(S_2'-(S_1')^2\rt)}\ge 0$ and $\mu'=
  S_1'-\frac{\lambda'}{k}$. In particular, $\log |A_i|\ge\mu'$ for all
  $i=1,\ldots,k$.

\end{lemma}
\begin{proof}
The proof is very similar to the proof of \lemref{lem:fixed_Qk}. The
only difference is that, in order to apply KKT, we use Slater's condition
instead (when $\beta = 0$, the nonlinear constraint becomes a convex
constraint). We also remark that, when $k=s$, we must have
$\alpha_1 = \alpha_2 = \ldots =\alpha_k = \frac{1}{k}$, and
the value of the objective function is $S_1'$.
\end{proof}

\noindent \textbf{Remark}. In \lemref{lem:fixed_Pk} it is possible 
that $\lambda' = 0$ (for instance, when all the $A_i$'s have
the same size).
\medskip

We now are ready to present the proof of Step 5. With slight abuse of
notation, we use $\valpha$ to express the sub-vector induced by the nonzero
coordinates of vector $\valpha$.

\begin{proposition}

  \label{prop:step_5}
  For any solution $\valpha$ to $\OPT(s)$ such that $\SUPP(\valpha)\in\mc{Q}$,
  we have $\SUPP(\valpha)\in \mc{Q}_\cei{s}$. In particular,
  $s\not\in \mathbb{Z}$.

\end{proposition}
\begin{proof}
Let $\valpha$ be a solution to $\OPT(s)$ such that $\SUPP(\valpha)\in\mc{Q}$.
We will show that in fact $\SUPP(\valpha)\in \mc{Q}_\cei{s}$. Suppose, towards
contradiction, that $\SUPP(\valpha)\in\mc{Q}_{k+1}$ for some $k\ge \cei{s}$.
Let the support of $\valpha$ be a collection $\{A_1,\ldots, A_k, A_{k+1},
A_1\cup A_{2}\}$ such that $A_1,\ldots, A_{k+1}$ form a $(k+1)$-partition of
$[q]$ and rewrite $\valpha=(\alpha_1,\ldots,\alpha_{k+1},\beta)>0$.

Fixed the collection $\{A_1, \ldots, A_{k+1}, A_1\cup A_2\}$, we
consider \eqref{opt:fixed_Qk} --- the restricted version of the \OPTprob.
Obviously, the vector $\valpha$ also achieves the maximum of
\eqref{opt:fixed_Qk}. Since $\valpha>0$, we may apply \lemref{lem:fixed_Qk}.
Note that $\lambda=k\sqrt{\frac{s}{k-s}\lt(S_2-S_1^2\rt)}$ and
$\mu=S_1-\frac{\lambda}{k}$, where
\begin{align*}
  S_1 &=\frac{1}{k}\lt[\log(|A_1|+|A_2|) + \sum_{i=3}^{k+1}\log|A_i|\rt],
      \Text{and} \\
  S_2 &=\frac{1}{k}\lt[\log^2(|A_1|+|A_2|)
    -2\log\lt(\frac{|A_1|+|A_2|}{|A_1|}\rt)
    \log\lt(\frac{|A_1|+|A_2|}{|A_2|}\rt)+\sum_{i=3}^{k+1} \log^2|A_i|\rt].
\end{align*}
From \lemref{lem:fixed_Qk}, we also see that $k>s$ (because $k+1\ne s+1$),
$\log|A_i|-\mu>0$ for $i\ge 3$, $\log\lt(\frac{|A_1|\cdot|A_2|}{|A_1|+
|A_2|}\rt)-\mu>0$ and
\[
  \OBJ(\valpha)=\mu+\frac{\lambda}{s}=S_1+\lt(\frac{1}{s}-
  \frac{1}{k}\rt)\lambda.
\]
We let $S_1'=S_1$ and
\[
  S_2'=\frac{1}{k}\lt[\log^2(|A_1|+|A_{2}|)+\sum_{i=3}^{k+1} \log^2|A_i|\rt].
\]
Clearly $S_2'>S_2$. We construct a vector $\valpha'\in \FEAS(s)$ such that
$\SUPP(\valpha')\in \mc{P}_k$ and $\OBJ(\valpha')>\OBJ(\valpha)$,
which is a contradiction to $\OBJ(\valpha)=\OPT(s)$. We let
$\valpha'=(\alpha_1',\ldots,\alpha_{k-1}', \alpha_k')$ with support being
the $k$-partition $\{A_1\cup A_2, A_3\ldots, A_k, A_{k+1}\}$ of $[q]$.
The coordinates of $\valpha'$ are defined by $\alpha_i'=\frac{\log|A_{i+1}|
-\mu'}{\lambda'}$ for $2\le i\le k$ and $\alpha_1'=\frac{\log(|A_1|+
|A_2|)-\mu'}{\lambda'},$ where
\[
  \lambda'=k\sqrt{\frac{s}{k-s}\lt(S_2'-(S_1')^2\rt)}
  \qquad\text{and}\qquad
  \mu'=S_1'-\frac{\lambda'}{k}.
\]
Note that $k>s$, $S_1'=S_1$ and $S_2'>S_2$. So $\lambda'>\lambda$
and $\mu'<\mu$, then $\log|A_{i+1}|-\mu'> \log|A_{i+1}|-\mu>0$
for $2 \le i \le k$ and $\log(|A_1|+|A_2|)-\mu'> \log\lt(
\frac{|A_1|\cdot|A_2|}{|A_1|+|A_2|}\rt)-\mu>0$, which implies that
$\valpha'>0$.
It is also not hard to verify that $\alpha_1'+\ldots+\alpha_k'=1$ and
$(\alpha_1')^2+\ldots+(\alpha_k')^2=\frac{1}{s}$, therefore indeed
$\valpha'\in \FEAS(s)$. Simplifying the expression of $\OBJ(\valpha')$,
we get
\[
  \OBJ(\valpha')=\mu'+\frac{\lambda'}{s}=S_1'+\lt(\frac{1}{s}-
    \frac{1}{k}\rt)\lambda',
\]
which is strictly larger than $\OBJ(\valpha)=\OPT(s)$. This contradiction
proves that $\SUPP(\valpha)\in \mc{Q}_{\cei{s}}$. To complete the
proof of the proposition, it remains to show that $s\not\in \mathbb{Z}$.
Since $\alpha_1+\alpha_2,\beta > 0$, by the remark in \eqref{eqn:Qk_bound_k}
we must have $\cei{s} = k > s$, hence $s$ is not an integer and we
finish the proof.
\end{proof}

\subsection{Step 6: establishing the structure of a solution}

We quote some results in \cite{Petr} that characterize graphs having at most two
nonnegative eigenvalues. Let $\mc{B}_{2\ell}$ be the class of graphs $G$
satisfying the following conditions:
\begin{itemize}
  \item $V(G)=X\cup Y$ with $X\cap Y=\emptyset$, where $X$ is a union of
        $\ell$ disjoint sets $X_1,\ldots,X_\ell$ of vertices and $Y$ is
        a union of $\ell$ disjoint sets $Y_1,\ldots,Y_\ell$ of vertices;
  \item $G[X]$ and $G[Y]$ are two complete subgraphs of $G$;
  \item for each $i,j\ge 2$, every vertex of $X_i$ is adjacent to every
        vertex of $Y_\ell\cup \ldots \cup Y_{\ell-i+2}$ and every vertex of
        $Y_j$ is adjacent to every vertex of $X_\ell\cup \ldots \cup
        X_{\ell-j+2}$.
\end{itemize}
Let $\mc{B}_1$ be the class of complete graphs. For $\ell\ge 1$,
let $\mc{B}_{2\ell+1}$ be the class of graphs $G$ for which
$V(G)=X\cup Y\cup Z$, $G[X\cup Y]\in \mc{B}_{2\ell}$ and every vertex of
$Z$ is adjacent to all other vertices of $G$. (See \figref{fig:bk_graphs}.)
\begin{lemma}[Lemma 4 in \cite{Petr}]

  \label{lem:bk}
  If a graph $G$ does not contain any of $3K_1$, $C_4$, $C_5$ as induced
  subgraphs then $G\in \bigcup_{t\ge 1}\mc{B}_t$.

\end{lemma}

\begin{figure}[ht!]
\centering
\parbox{2in}{%
\centering
\begin{tikzpicture}
  [scale=0.5, auto=right, every node/.style={circle, draw, fill=white,
   inner sep=2pt, minimum width=18pt}]
  \node (q3) at (2 cm, 1 cm) {$q_3$};
  \node (q2) at (2 cm, -1 cm) {$q_2$};
  \node (q1) at (2 cm, -3 cm) {$q_1$};
  \node (p1) at (-2 cm, 3 cm) {$p_1$};
  \node (p2) at (-2 cm, 1 cm) {$p_2$};
  \node (p3) at (-2 cm, -1 cm) {$p_3$};

  \foreach \from/\to in {p2/q3, p3/q2, p3/q3, p1/p2, p2/p3, q1/q2, q2/q3}
    \draw (\from) -- (\to);

  \path[every edge/.style={draw}]
    (p1) edge[out=225,in=135] (p3)
    (q3) edge[out=-45,in=45] (q1);

\end{tikzpicture}
\caption*{$\mc{B}_6$}
}
\parbox{2in}{%
\centering
\begin{tikzpicture}
  [scale=0.5, auto=right, every node/.style={circle, draw, fill=white,
   inner sep=2pt, minimum width=18pt}]
  \node (q3) at (2 cm, 1 cm) {$q_3$};
  \node (q2) at (2 cm, -1 cm) {$q_2$};
  \node (q1) at (2 cm, -3 cm) {$q_1$};
  \node (p1) at (-2 cm, 3 cm) {$p_1$};
  \node (p2) at (-2 cm, 1 cm) {$p_2$};
  \node (p3) at (-2 cm, -1 cm) {$p_3$};
  \node (r) at  (2 cm, 3 cm) {$r$};

  \foreach \from/\to in {p2/q3, p3/q2, p3/q3, p1/p2, p2/p3, q1/q2,
                         q2/q3, r/p1, r/p2,r/p3,r/q3}
    \draw (\from) -- (\to);

  \path[every edge/.style={draw}]
    (p1) edge[out=225,in=135] (p3)
    (q3) edge[out=-45,in=45] (q1)
    (r) edge[out=-45,in=45] (q2)
    (r) edge[out=-40,in=40] (q1);

\end{tikzpicture}
\caption*{$\mc{B}_7$}
}
\parbox{2.5in}{%
\centering
\begin{tikzpicture}
  [scale=0.5, auto=right, every node/.style={circle, draw, fill=white,
   inner sep=2pt, minimum width=18pt}]
  \node[anchor=east] (endp1) at (2,1.2) {$p$};
  \node[anchor=west,draw=none, fill=none] (desc1) at (2.3,1.2)
    {$=$ Complete graph $K_p$.};
  \node[anchor=east] (endp21) at (0,-0.8) {$p$};
  \node[anchor=east] (endp22) at (2,-0.8) {$q$};
  \node[anchor=west,draw=none, fill=none] (desc2) at (2.3,-0.8)
    {$=$ Complete graph $K_{p+q}$.};

  \path[every edge/.style={draw}]
    (endp21) edge (endp22);

\end{tikzpicture}
}
\caption{Some examples of graphs in $\mc{B}_6 \cup \mc{B}_7$.}
\label{fig:bk_graphs}
\end{figure}

We now present the proof of Step 6.

\begin{proposition}

  \label{prop:step_6}
  For any solution $\valpha$ to $\OPT(s)$ not having induced copies of
  $3K_1$ and $C_4$ in   its support graph $\SUPP(\valpha)$, it is true
  that $\SUPP(\valpha)\in \lt(\cup_{k\ge\cei{s}}\mc{P}_k\rt)\cup
  \mc{Q}_\cei{s}$.

\end{proposition}
\begin{proof}
Let $\valpha$ be an solution to $\OPT(s)$ such that $\SUPP(\valpha)$
contains neither $3K_1$ nor $C_4$ as an induced subgraph. We know, by
\propref{prop:step_1} (Step 1), that $\SUPP(\valpha)$ does not
contain $C_5$ as induced subgraph as well. By \lemref{lem:bk},
$\SUPP(\valpha)\in \bigcup_{t\ge 1}\mc{B}_t$. We will show that
$\SUPP(\valpha) \in \mc{P} \cup \mc{Q}_\cei{s}$.

First we show that $t\le 5$. Suppose not. Then $\ell:=\flo{t/2}\ge 3$
and there exists an induced subgraph $G$ of $\SUPP(\valpha)$ such that
$G\in \mc{B}_{2\ell}$. Recall the definition of $\mc{B}_{2\ell}$. We choose
four vertices $x_1$,$x_2$,$y_1$,$y_2$ of $G$ such that $x_1\in X_1$,
$x_2\in X_2$, $y_1\in Y_1$, $y_2\in Y_2$, then $\{x_1,x_2,y_1,y_2\}$ induces
a matching of size two in $\SUPP(\valpha)$, contradicting
\propref{prop:step_2} (Step 2).

We claim that $1\le t\le 3$. Suppose not, then either $t=4$ or $t=5$.
There exists an induced subgraph $G$ of $\SUPP(\valpha)$ such that
$G\in \mc{B}_{4}$. So $V(G)=X_1\cup X_2\cup Y_1\cup Y_2$. If there
exist two vertices $x_1,x_2$ in $X_1$, then $\{x_1,x_2,y_1,y_2\}$ induces
a matching of size two in $\SUPP(\valpha)$ for any $y_1\in Y_1, y_2\in Y_2$,
contradicting \propref{prop:step_2} again.
Therefore, we may assume that $X_1=\{A\}$ and by symmetry $Y_1=\{B\}$ for
$A,B\subseteq [q]$. Let $X_2=\{C_1,\ldots,C_l\}$ and $Y_2=\{D_1,\ldots,D_m\}$,
where $C_i,D_j\subseteq [q]$. Since $X_2 \cup Y_2$ induces a clique, the
sets $\bigcup_{i=1}^l C_i$ and $\bigcup_{j=1}^m D_j$ are disjoint.
By similar reasons, $A\cap \lt(\bigcup_{i=1}^l C_i\rt)= \emptyset$ and
$B \cap \lt(\bigcup_{j=1}^m D_j\rt)= \emptyset$. \propref{prop:supp_inter}
implies that $A\cup C_1\cup\ldots\cup C_l=B\cup D_1\cup\ldots
\cup D_m$. Then clearly $C_1\subseteq B\setminus A \subset B$
and $D_1\subseteq A\setminus B \subset A$. Now, $\{A,C_1,D_1,B\}$
induces a path $A-C_1-D_1-B$ in $\SUPP(\valpha)$ with $|A|>|D_1|$
and $|B|>|C_1|$, contradicting \propref{prop:step_3} (Step 3).

If $t=1$ then it is easy to see that $\SUPP(\valpha)\in \mc{P}$
and by \eqref{eqn:Qk_bound_k} we must have $\SUPP(\valpha)\in
\bigcup_{k\ge \cei{s}} \mc{P}_k$. Otherwise, we may assume that
$\SUPP(\valpha)\in \mc{B}_2 \cup \mc{B}_3$. The support of $\valpha$
can be expressed as a union $X\cup Y\cup Z$ (possibly with $Z=\emptyset$)
such that each of $X\cup Z$ and $Y\cup Z$ induces a clique in
$\SUPP(\valpha)$ and there is no edge between $X$ and $Y$.

If $|X|\ge 2$ and $|Y|\ge 2$, then $\SUPP(\valpha)$ has an induced matching of
size two, which cannot happen by \propref{prop:step_2} (Step 2).
Thus, we may assume that $|Y|=1$ and thereby $\SUPP(\valpha)$ can be
viewed as a graph obtained from a clique by removing a star. Now
\propref{prop:step_4} (Step 4) and \propref{prop:step_5} (Step 5)
together imply that $\SUPP(\valpha)\in \mc{Q}_\cei{s}$. This
finishes the proof of Step 6.
\end{proof}

\begin{corollary}

  \label{cor:step_6}
  There exists a solution $\valpha$ to $\OPT(s)$ such that $\SUPP(\valpha)\in
  \lt(\cup_{k\ge\cei{s}}\mc{P}_k\rt)\cup \mc{Q}_\cei{s}$.

\end{corollary}

\subsection{Step 7: the structure of all solutions}

In this subsection we characterize all solutions to the \OPTprob.
Namely, we prove that not only we can find a solution $\valpha$ for which
$\SUPP(\valpha)\in \mc{P}\cup \mc{Q}_\cei{s}$ as \corref{cor:step_6}
asserts, but in fact, all solutions to $\OPT(s)$ necessarily satisfy
$\SUPP(\valpha)\in \mc{P}\cup \mc{Q}_\cei{s}$.

Among the first six steps in the proof of \thmref{thm:structure},
Step 1 is the only step that fails to assert a property that all solutions
necessarily enjoy, as it does not prevent the existence of a solution
$\valpha$ such that $\SUPP(\valpha)$ contains either $3K_1$ or $C_4$ as
induced subgraphs. To remedy this situation, we prove the following
three propositions.

\begin{proposition}

  \label{prop:cant_walk_3K1}
  If $\valpha$ is a solution to $\OPT(s)$ and $\valpha'$
  is another solution to $\OPT(s)$ obtained from $\valpha$ by an application of
  \propref{prop:3_eigenvalues} with $H$ being isomorphic to
  $3K_1$, then $\SUPP(\valpha') \not\in \mc{P}\cup \mc{Q}$.

\end{proposition}
\begin{proof}
Suppose, towards contradiction, that $\SUPP(\valpha')\in \mc{P}\cup
\mc{Q}$. Let $A$, $B$ and $C$ be the vertices of $H$. Clearly $A$, $B$,
$C$ pairwise intersect, because they induce an independent set of size three.
Moreover, we have $\SUPP(\valpha)-\{A,B,C\}\subset \SUPP(\valpha')
\subset \SUPP(\valpha)$.
Since $\SUPP(\valpha') \in \mc{P}\cup \mc{Q}$, we know that some of
the sets in $\SUPP(\valpha')$ form partition of $[q]$.
Let $\mc{R}$ be such a partition with maximum number of sets. Because
$\mc{R}$ can contain at most one element from $\{A,B,C\}$ and $\alpha'_A
+\alpha'_B+\alpha'_C=\alpha_A+\alpha_B+\alpha_C > 0$, we may
assume, without loss of generality, that $B,C \not\in \mc{R}$ and $A$
is in the support of $\valpha'$. We remark that we do not necessarily
have $A\in \mc{R}$. Since $B\cap C \ne\emptyset$, there exists
a set $A'\in \mc{R}$ such that $A'\cap B \cap C \ne\emptyset$.

We claim that $X\subseteq B$ for any subset $X\in\mc{R}$
intersecting $B$. To prove this, we use \propref{prop:supp_inter}.
If $X\nsubseteq B$, then there exists a set $Y$ in the support of $\valpha$
which is disjoint from $B$ and intersects $X$. The set $Y$ must be in
the support of $\valpha'$ as well, since $Y\not\in\{A,B,C\}$. Because
$X\cap Y \ne \emptyset$, we must have that $\SUPP(\valpha')\in\mc{Q}$,
which implies $X \subseteq Y$ (recall $X\in \mc{R}$). This is a
contradiction, because $Y$ must be disjoint from $B$, thereby proving
the claim. Therefore, there exists a subfamily $\mc{R}_1\subset \mc{R}$
such that $B=\cup_{X\in \mc{R}_1} X$.
By switching the roles of $B$ and $C$ in the previous argument,
we conclude that there exists $\mc{R}_2\subset \mc{R}$ such that
$C=\cup_{X\in \mc{R}_2} X$. As $B,C\notin \mc{R}$, we see that
$|\mc{R}_1|\ge 2$ and $|\mc{R}_2|\ge 2$;
as $A'\cap B\cap C\ne \emptyset$, we must have $A' \subseteq B\cap C$.
We assume, from now on, that $|B| \ge |C|$.

We have two cases to consider:

\begin{enumerate}[(i) -]
\item $C\nsubseteq B$. In this case, we have $B-C\ne \emptyset$ and
$C-B\ne \emptyset$, so there exist two sets $D,E\in\mc{R}\setminus \{A'\}$
such that $D \subseteq B$, $E\subseteq C$, and $D\cap C = E \cap B =
\emptyset$; or
\item $C\subset B$. Since $|\mc{R}_2|\ge 2$, there exists a set
$D\in \mc{R}\setminus \{A'\}$ such that $D\subseteq B \cap C$.
\end{enumerate}
In the first case (i), the sets $B,C,D,E$ in the support of $\valpha$
induce a path $C-D-E-B$ satisfying $|C| > |E|$ and $|B| > |D|$.
But this is forbidden by \propref{prop:step_3} (Step 3).
In the second case (ii), if $H'$ is the subgraph induced by
the sets $A', B, C, D$, then the adjacency matrix of $H'$
with respect to the sets $A', B, C, D$ (in that order) is
\[
  M= \begin{pmatrix}
    0 & 0 & 0 & 1 \\
    0 & 0 & 0 & 0 \\
    0 & 0 & 0 & 0 \\
    1 & 0 & 0 & 0
  \end{pmatrix}.
\]
The matrix $M$ has three nonnegative eigenvalues and its
kernel is spanned by the vectors $(0,1,0,0)$ and $(0,0,1,0)$.
However, there is no vector $\mathbf{v}$ in the kernel of $M$ such
that $\mathbf{v}$ is perpendicular to both
\[
  (1, 1, 1, 1)\quad\text{and}\quad
  (\log |A'|, \log |B|, \log |C|, \log |D|),
\]
since $\log|B| > \log |C|$ in (ii). Hence, by \obsref{obs:obs_kernel},
$\valpha$ cannot be an optimal solution to $\OPT(s)$, finishing
the proof of the proposition.
\end{proof}

\begin{proposition}

  \label{prop:cant_walk_C5plus}
  If $\valpha$ is a solution to $\OPT(s)$ and $\valpha'$
  is another solution to $\OPT(s)$ obtained from $\valpha$ by an application of
  \propref{prop:3_eigenvalues} with $H$ being isomorphic to
  $C_5^+$, then $\SUPP(\valpha') \not\in \mc{P}\cup \mc{Q}$.

\end{proposition}
\begin{proof}
Suppose, towards contradiction, that $\SUPP(\valpha')\in \mc{P}\cup
\mc{Q}$. Let sets $A$, $B$, $C$, $D$ and $E$ induce the subgraph $H$
(see \figref{fig:C5plus}), whose adjacency matrix (with respect to the
order of $A,B,C,D,E$) is
\[
  M= \begin{pmatrix}
    0 & 1 & 0 & 1 & 1 \\
    1 & 0 & 1 & 0 & 1 \\
    0 & 1 & 0 & 1 & 0 \\
    1 & 0 & 1 & 0 & 0 \\
    1 & 1 & 0 & 0 & 0
  \end{pmatrix}.
\]
The matrix $M$ has three nonnegative eigenvalues and has kernel spanned by
the vector $(-1, 1, 1, -1, 0)$. Therefore, the vector $\vgamma$ yielded by
the proof of \propref{prop:3_eigenvalues} for $H$ has the form
$\vgamma=(\gamma_A,\gamma_B,\gamma_C,\gamma_D, 0)$, where
$\gamma_A = \gamma_D=-\gamma_B=-\gamma_C$. Then either $C$ or $D$ is
in the support of $\valpha'$ (since either $\gamma_C\ge 0$ or
$\gamma_D \ge 0$); by the symmetry between $C$ and $D$, we may
assume that $C\in \SUPP(\valpha')$. We point out that
$E\in \SUPP(\valpha')$, as the coordinate $\valpha_E$ is not changed.
In addition, we have $C\cap E\ne \emptyset$. This implies that
$\SUPP(\valpha')\in \mc{Q}$; and moreover, either $C\subset E$
or $E\subset C$. If $C\subset E$, together with the fact that
$A\cap E=\emptyset$, we get that $A\cap C=\emptyset$, contradicting
our definition of $C_5^+$. If $E\subset C$, plus $C\cap D=\emptyset$,
we get that $E\cap D=\emptyset$, again contradicting the definition of $C_5^+$.
This finishes the proof.
\end{proof}

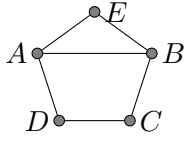
\begin{figure}[ht!]
\centering
\begin{tikzpicture}
  [scale=0.8, auto=left, every node/.style={circle, draw, fill=black!50,
   inner sep=0pt, minimum width=4pt}]
  \node (n1) at (18:1 cm) [label=right:$B$]{};
  \node (n2) at (90:1 cm) [label=right:$E$]{};
  \node (n3) at (162:1 cm) [label=left:$A$]{};
  \node (n4) at (234:1 cm) [label=left:$D$]{};
  \node (n5) at (306:1 cm) [label=right:$C$]{};

  \foreach \from/\to in {n1/n2, n2/n3, n3/n4, n4/n5, n5/n1,n1/n3}
    \draw (\from) -- (\to);
\end{tikzpicture}
\caption{graph $C_5^+$}
\label{fig:C5plus}
\end{figure}

\begin{proposition}

  \label{prop:cant_walk_C4}
  Let $\valpha$ be a solution to $\OPT(s)$ and $\valpha'$
  be another solution to $\OPT(s)$ obtained from $\valpha$ by an application of
  \propref{prop:3_eigenvalues} with $H$ being isomorphic to
  $C_4$. If $\SUPP(\valpha)$ does not contain an induced copy
  of $3K_1$ or $C_5^+$, then $\SUPP(\valpha') \not\in \mc{P}\cup \mc{Q}$.

\end{proposition}
\begin{proof}
Suppose, towards contradiction, that $\SUPP(\valpha')\in \mc{P}\cup
\mc{Q}$.  Let $A$, $B$, $C$ and $D$ be the vertices of $H$, where we
assume the pairs $\{A, B\}$, $\{B,C\}$, $\{C,D\}$ and $\{D,A\}$ induce
edges in $\SUPP(\valpha)$, or equivalently, these are the pairs of disjoint
sets. Let $\vgamma\in\R^4$ be the vector that the proof of
\propref{prop:3_eigenvalues} yields for $H$. By \obsref{obs:obs_kernel},
we have $M\cdot \vgamma = 0$, where
\[
  M= \begin{pmatrix}
    0 & 1 & 0 & 1 \\
    1 & 0 & 1 & 0 \\
    0 & 1 & 0 & 1 \\
    1 & 0 & 1 & 0
  \end{pmatrix}.
\]
is the adjacency matrix of $H$, so $\gamma_A+\gamma_C=\gamma_B+\gamma_D=0$.
As in the proof of \propref{prop:cant_walk_3K1}, denote by $\mc{R}$ the
partition of $[q]$ in $\SUPP(\valpha')$ with maximum number of sets.
We remark that among the vertices of $V(H)$, at least one but at most two
belong to $\mc{R}$. This is because the elements of $\mc{R}$ induce
a clique in $\SUPP(\valpha')$. We may then assume, without loss of
generality, that $A\in\mc{R}$ and $C,D\not\in \mc{R}$.
We claim that $B\in \mc{R}$. If not, then because $B\cap D
\ne\emptyset$, there exists a set $B'\in\mc{R}$ such that
$B' \cap (B\cap D)\ne\emptyset$. But if that is the case, the sets
$B', B, D$ would induce a copy of $3K_1$ in $\SUPP(\valpha)$, which is
clearly a contradiction. Thus, we must have $B\in \mc{R}$.

We point out that there is no set $E\in \mc{R}$ such that
$E\cap C\ne\emptyset$ and $E\cap D\ne \emptyset$, as, otherwise,
the sets $A,B,E,C,D$ would induce a subgraph of $\SUPP(\valpha)$
isomorphic to $C_5^+$, which is forbidden.

We now claim that $C\setminus A$ and $D \setminus B$ are nonempty.
Assume, for contradiction, that $C \subset A$. By
\propref{prop:supp_inter}, since $A$ is not a subset of $C$, there
exists $X$ in the support of $\valpha$ which intersects $A$ but is
disjoint from $C$. Since $X\notin \{A,B,C,D\}$, $X$ must belong to
$\SUPP(\valpha')$ as well.
But $X$ intersects $A$, and $A\in \mc{R}$,
hence we conclude that $\SUPP(\valpha')\in \mc{Q}$ and $A\subseteq X$.
Since $X$ is disjoint from $C$, we see $A$ is also disjoint from $C$,
which is a contradiction to the definition of $H$.
Therefore $C$ is not a proper subset of $A$, and similarly,
$D$ is not a proper subset of $B$.

The previous proved facts imply that there exist two sets
$A',B'\in \mc{R}\setminus\{A,B\}$ such that $A,A',B,B'$ are disjoint, $A'\cap C =
B' \cap D = \emptyset$, while $A'\cap D$ and $B'\cap C$ are both nonempty.
For instance, take $A'\in \mc{R}$ which intersects $D\setminus B$
and $B'\in \mc{R}$ which intersects $C\setminus A$. The adjacency matrix
of the subgraph of $\SUPP(\valpha)$ induced by $A,B',A',B, C, D$ (in that order)
is given by
\[
  M'= \begin{pmatrix}
    0 & 1 & 1 & 1 & 0 & 1 \\
    1 & 0 & 1 & 1 & 0 & 1 \\
    1 & 1 & 0 & 1 & 1 & 0 \\
    1 & 1 & 1 & 0 & 1 & 0 \\
    0 & 0 & 1 & 1 & 0 & 1 \\
    1 & 1 & 0 & 0 & 1 & 0
  \end{pmatrix}.
\]
But $M'$ has three positive eigenvalues, which is forbidden
by \propref{prop:3_eigenvalues}. This final contradiction
establishes the proposition.
\end{proof}

Propositions~\ref{prop:cant_walk_3K1},~\ref{prop:cant_walk_C5plus},
and~\ref{prop:cant_walk_C4} together with Steps 1 through 6
imply \thmref{thm:structure}.
\begin{proof}[Proof of \thmref{thm:structure}]
Suppose towards contradiction that there exists a solution
$\vbeta$ to $\OPT(s)$ for which the support $\SUPP(\vbeta)$
belongs to neither $\mc{P}$ nor $\mc{Q}_{\cei{s}}$.
From the proof of Step 1, by repeatedly applying \propref{prop:3_eigenvalues},
we can find a solution $\valpha'$ such that
$\SUPP(\valpha')$ has no induced $3K_1$ and $C_4$ and $\vbeta$
and $\valpha'$ are connected by a piecewise linear path in
$\FEAS(s)$. From Step 6, we have $\SUPP(\valpha')\in \mc{P}
\cup \mc{Q}_{\cei{s}}$. Let $\valpha$ be the last node in the piecewise linear
path from $\vbeta$ to $\valpha'$ before reaching the endpoint $\valpha'$.
Clearly $\valpha'$ was obtained from $\valpha$ by an application of
\propref{prop:3_eigenvalues} with $H$ being isomorphic to
one of $3K_1$, $C_5^+$ or $C_4$. As remarked in the proof of
\propref{prop:step_1}, if $H$ is isomorphic to $C_4$, we may
further assume that $\SUPP(\valpha)$ has no induced copy of $3K_1$
or $C_5^+$. From Propositions~\ref{prop:cant_walk_3K1},%
~\ref{prop:cant_walk_C5plus}~and~\ref{prop:cant_walk_C4}, we conclude
that $\SUPP(\valpha')\not\in \mc{P}\cup \mc{Q}$, a contradiction.

If $\SUPP(\vbeta)\in \mc{P}$, since the sizes of all sets in the support
of $\vbeta$ are at least one and add up to $q$, it is clear that
$\SUPP(\vbeta)\in \cup_{\cei{s}\le k\le q}\mc{P}_k$. In the case that
$\cei{s}<q$, we have $\SUPP(\vbeta)\notin \mc{P}_q$, as otherwise
$\OBJ(\vbeta)=0$ which can not be the optimal objective value.
This establishes \thmref{thm:structure}.
\end{proof}

%% file: applications.tex
\section{First applications from the structure}
\label{sec:applications}
In this section, we provide short proofs to two results with the aid
of \thmref{thm:structure}. We first consider the structure of extremal
graphs with $n$ vertices and $m\le n^2/4$ edges which maximizes the
number of $3$-colorings. Such extremal graphs were conjectured to be
close to complete bipartite graphs plus some isolated vertices by
Lazebnik~\cite{Laz89}, and this was later confirmed in \cite{LPS}. Here we
present an asymptotic version by a rather simple proof.

\begin{theorem}

  For any $\eps>0$, the following holds for all sufficiently large $n$.
  Let $G$ be an $n$-vertex with $m\le n^2/4$ edges which maximizes the
  number of $3$-colorings. Then there exists an $n$-vertex graph $G_0$,
  which is a complete bipartite graph plus some isolated vertices, such
  that $G$ is $\eps n^2$-close to $G_0$.

\end{theorem}
\begin{proof}
Apply \corref{cor:stability-LPS} to $G$ and $s=\frac{n^2}{n^2-2m}$. Then $G$
is $\eps n^2$-close to some $G_\valpha(n)$, where $\valpha$ solves
$\textsc{OPT}_3(s)$. By \thmref{thm:structure} and $\cei{s}=2$, we get
$\textsc{SUPP}_3(\valpha)\in \mc{Q}_2\cup \mc{P}_2$.
This implies that $G_\valpha(n)$ is a complete bipartite graph plus
some isolated vertices, finishing the proof.
\end{proof}

Next, we prove \thmref{thm:q=s+C} (i). As mentioned in \secref{sec:intro},
the case $q=s+1$ was first proved in \cite{LPS} and then extended to all
$n$ in \cite{LazTofts}. We need the following convenient definition.
When $s\ge 1$ is integer, we define the \emph{$s$-balanced vector for $q$}
to be a vector $\valpha$ such that $\alpha_{A_i}=\frac{1}{s}$ and $A_1,
\ldots,A_s$ forms a balanced $s$-partition of $[q]$.

\begin{proof}[Proof of \thmref{thm:q=s+C} (i)]
Let us only consider when $s+1\le q\le s+2$.
In view of \corref{cor:stability-LPS}, it suffices to show that for every
$s\ge 2$, the $s$-balanced vector $\valpha$ for $q$ is the unique solution
to $\OPT(s)$. By \thmref{thm:structure} and \propref{prop:step_5}, we see
this indeed is the case for $q=s+1$.

Now assume $q=s+2$. It is easy to compute that $\OBJ(\valpha)=
\frac{2}{s}\log 2$. Suppose that there exists a solution $\vbeta$ to
$\OPT(s)$, which is not $\valpha$. By \thmref{thm:structure} and
\propref{prop:step_5}, $\SUPP(\vbeta)\in \mc{P}_{s+1}$. Then we may
assume that the sets $A_1,A_2,\ldots, A_{s}$ in the support of $\vbeta$
are of size 1, except $A_{s+1}$, which is of size $2$. Using
\lemref{lem:fixed_Pk}, we compute that $\lambda'=s\log 2$ and
$\mu'=-\frac{s-1}{s+1}\log 2$, which implies that $\OBJ(\vbeta)=
\frac{2}{s+1}\log 2< \OBJ(\valpha)$, a contradiction. This completes
the proof.
\end{proof}

%% file: counterexample.tex
\section{The counterexamples to Lazebnik's conjecture}
\label{sec:counterexample}
In this section, we show various counterexamples to \conjref{conj:Lazebnik}.
First, we prove \thmref{thm:q=s+C} (ii).

\begin{proof}[Proof of \thmref{thm:q=s+C} (ii)]
We first show that the $10$-balanced vector $\valpha$ for $q=13$
has a smaller objective value than a vector $\vbeta \in \textsc{FEAS}_{13}(10)$
with $11$ color classes, which together with \thmref{thm:pob}
implies that \conjref{conj:Lazebnik} is false for $(s,q)=(10,13)$.
It is easy to compute that $\textsc{OBJ}_{13}(\valpha)=\frac{3}{10}\log 2$.
Let $\vbeta:=(\beta_1,\beta_2,\ldots,\beta_{11})$, where
$\beta_1=\beta_2=\frac{1}{11}+\frac{3\sqrt{5}}{110}$
and $\beta_3=\ldots=\beta_{11}=\frac{1}{11}-\frac{\sqrt{5}}{165}$,
and define the sizes of the color classes by $A_1=A_2=2$ and
$A_3=\ldots=A_{11}=1$. So indeed we have $\vbeta \in
\textsc{FEAS}_{13}(10)$. Now it is easy to verify that
$\textsc{OBJ}_{13}(\vbeta)=(\beta_1+\beta_2)\log 2=
(\frac{2}{11}+\frac{3\sqrt{5}}{55})\log 2>\frac{3}{10}\log 2=
\textsc{OBJ}_{13}(\valpha)$.

We claim that \conjref{conj:Lazebnik} is false for the pair $(s,q)$,
provided the $s$-balanced vector $\vgamma$ for $q$ contains at least
three color classes of size 2 and seven color classes of size 1.
In this case, $\vgamma$ contains a subvector $\frac{10}{s} \valpha$
formed by these $10$ color classes.
Define a new vector $\vgamma'$ obtained from $\vgamma$ by replacing
$\frac{10}{s} \valpha$ with $\frac{10}{s}\vbeta$. It is easy to see
that $\V(\vgamma')=\V(\vgamma)=1$ and $\E(\vgamma')=\E(\vgamma)=
\frac{s-1}{2s}$, so $\vgamma'\in \FEAS(s)$. Since the objective value
of $\vbeta$ is bigger than $\valpha$, it holds that
$\OBJ(\vgamma')>\OBJ(\vgamma)$, proving the claim.

From the claim we obtain that \conjref{conj:Lazebnik} is false for all $(s,q)$,
provided $s+3\le q\le 2s-7$ and $s\ge 10$, finishing the proof.
\end{proof}

The next result gives us more counterexamples in a wider range,
which gives rise to the proof of \thmref{thm:counterexample}.

\begin{theorem}

  \label{thm:counterex-str}
  If $s,t,r$ are integers such that $t\ge 2$ and $50 t \log t \le r
  \le \min\{\frac{s}{2}, \frac{3}{2}t^2 \log^2 t\}$,
  then Lazebnik's conjecture is false for $(s,q)$, where $q:= st + r$.

\end{theorem}
\begin{proof}
It is not hard to see that the objective value of the $s$-balanced vector for $q$
(corresponding to the Tur\'an graph $T_s(n)$, which is conjectured by Lazebnik to be optimal) is
\[
  X:= \frac{r}{s} \log (t+1) + \frac{s-r}{s} \log t.
\]
We will now construct a solution $\valpha$ with $s+1$ parts which yields a
larger objective value. We begin by defining the sequence of the sizes of
the color classes: let $A_1=\ldots =A_{r-1}=t+1$, $A_r=\ldots=A_s=t$ and
$A_{s+1}=1$. And let $\valpha=(\alpha_1,\ldots,\alpha_{s+1})$ such that
\[
  \alpha_i := \frac{\log A_i - S_1}{(s+1)\sqrt{s \cdot (S_2 - S_1^2)}}
  + \frac{1}{s+1},
\]
where $S_1 := \frac{1}{s+1} \sum_{i=1}^{s+1} \log A_i$ and
$S_2 := \frac{1}{s+1} \sum_{i=1}^{s+1} \log^2 A_i$. Observe that
\[
  S_2 - S_1^2 =
  \frac{s}{s+1} \lt[ S_2' - (S_1')^2\rt] + \frac{s}{(s+1)^2} (S_1')^2 > 0,
\]
where $S_1':= \frac{1}{s} \sum_{i=1}^{s} \log A_i$ and
$S_2' := \frac{1}{s} \sum_{i=1}^{s} \log^2 A_i$, thus
the $\alpha_i$'s are well-defined real numbers. Moreover
\[
  S_2' - (S_1')^2 = \frac{(r-1)(s-r+1)}{s^2}\log^2\lt(1+\frac{1}{t}\rt).
\]
We claim that $\alpha_i \ge 0$. This is because $\log A_i\ge 0$ and
$\sqrt{s \cdot (S_2 - S_1^2)} \ge \sqrt{\frac{s^2}{(s+1)^2} (S_1')^2} = S_1.$
One can check that $\sum_{i=1}^s \alpha_i = 1$ and
$\sum_{i=1}^{s+1}\alpha_i^2 = \frac{1}{s}$, hence $(\valpha, \vA)$ is a feasible
solution, having objective value
\[
 Y:= \sum_{i=1}^{s+1} \alpha_i \log A_i = S_1 + \sqrt{\frac{S_2-S_1^2}{s}}.
\]
To finish the proof of the theorem, we just need to show that $Y > X$.
Observe that $X = S_1' + \frac{\log\lt(1+\frac{1}{t}\rt)}{s}$.
Let $M:=\frac{S_2' - (S_1')^2}{s+1}$ and $N:=\frac{S_1'}{s+1}$. We have
$M \le 3N^2$, because
\[
 M = \frac{(r-1)(s-r+1)}{s^2(s+1)}\log^2\lt(1+\frac{1}{t}\rt)
   \le \frac{r}{s^2 t^2} \le \frac{3 \log^2 t}{2s^2} \le
   \frac{3(S_1')^2}{(s+1)^2} = 3N^2,
\]
where we used $0 < \log\lt(1+\frac{1}{t}\rt) < \frac{1}{t}$ and
$r \le \frac{3}{2}t^2 \log^2 t$. From $M \le 3N^2$ we obtain
$\sqrt{N^2 + M} \ge N + \frac{M}{3N}$, thus
\[
 Y= S_1 + \sqrt{N^2 + M} \ge S_1 + N + \frac{M}{3N} =
 S_1' + \frac{M}{3N}.
\]
The only step left is to verify that $\frac{M}{3N} >
\frac{\log\lt(1 + \frac{1}{t}\rt)}{s}$, which can be done as follows
\begin{align*}
  \frac{M}{3N}=\frac{S_2'-(S_1')^2}{3 S_1'} >
  \frac{(r-1)(s-r+1)}{6 s^2 \log t}\log^2\lt(1+\frac{1}{t}\rt)
 \ge \frac{r}{50s\log t}\cdot \frac{\log\lt(1 + \frac{1}{t}\rt)}{t}
   \ge\frac{\log\lt(1 + \frac{1}{t}\rt)}{s},
\end{align*}
where we used $S_1' \le \log(t+1) < 2 \log t$, $s-r+1 \ge \frac{s}{2}$,
 $\log\lt(1+\frac{1}{t}\rt) \ge \frac{1}{2t}$, and $r\ge 50t \log t$.
\end{proof}

\begin{proof}[Proof of \thmref{thm:counterexample}.]
We point out that for any integer $t$ satisfying $20 \le t \le
\frac{s}{200\log s}$, there always exists such an integer $r$ which satisfies
$50 t \log t \le r \le \min\{\frac{s}{2}, \frac{3}{2}t^2 \log^2t\}$.
Therefore (provided $s\ge 50000$) for any $20s\le q_0 \le
\frac{s^2}{200\log s}$ there exists an integer $q$ within distance at
most $s$ from $q_0$ such that Lazebnik's conjecture is false for $(s,q)$.
\end{proof}

%% file: integer.tex
\section{Solving $\OPT(s)$ for integer $s$}
\label{sec:opt_integer_s}

When $s\ge 1$ is integer, recall the definition of the $s$-balanced vector
for $q$ in Section \ref{sec:applications}. Because $x\mapsto \log x$ is concave, we know that, among all
candidate solutions with support graph in $\mc{P}_s$, the $s$-balanced
vector has the largest objective value. In this section we shall show
that this vector is indeed the unique solution to $\OPT(s)$ for an
integer $s$, provided that $q=\Omega\lt(\frac{s^2}{\log s}\rt)$.

\begin{theorem}

  \label{thm:Turan_OPT}
  For large enough integer $s$ and for $q\ge 100\frac{s^2}{\log s}$
  the $s$-balanced vector for $q$ is the unique solution to $\OPT(s)$.

\end{theorem}

\noindent Note that when $s$ is integer, the optimization problem $\OPT(s)$
corresponds to the problem of maximizing the number of $q$-colorings over the
family of graphs containing the Tur\'{a}n graph $T_s(n)$, i.e., Lazebnik's
\conjref{conj:Lazebnik}. In the light of \corref{cor:stability-LPS}, 
the establishment of \thmref{thm:Turan_OPT} gives rise to \thmref{thm:main}.

To prove \thmref{thm:Turan_OPT} we use the structural information from
\thmref{thm:structure}, that is, the support graph of a solution necessarily
belongs to $\lt(\cup_{\cei{s}\le k\le q}\mc{P}_k\rt)\cup \mc{Q}_{\cei{s}}$.
Additionally, we know by \propref{prop:step_5} that the support graph
is not in $\mc{Q}_{\cei{s}}$ (since $s$ is integer), hence we shall only
study candidate solutions $\valpha$ whose support forms a partition
of $[q]$. So proving \thmref{thm:Turan_OPT} amounts to showing
that the solution of $\OPT(s)$ lies in $\mc{P}_s$. With that in mind,
the optimization problem~\eqref{opt:fixed_Qk} for $\beta=0$ and variable
$A_i$'s can be stated as follows. Fix $\cei{s} \le k \le q$ and
\begin{equation}
  \label{opt:rephrased}
  \begin{array}{ll}
    \text{Maximize}   & \sum_{i=1}^k\alpha_i \log A_i \\
    \text{Subject to} & \sum_{i=1}^k\alpha_i=1,
                        \sum_{i=1}^k\alpha_i^2\le \frac{1}{s},
                        \sum_{i=1}^k A_i = q,\\
                      & \alpha_i\ge 0 \enskip\text{is real},
                        A_i>0 \enskip\text{is integer, for all}
                        \enskip i=1,\ldots, k.
  \end{array}
\end{equation}

It turns out that \eqref{opt:rephrased} becomes much simpler to analyze
when we relax the conditions on the $A_i$'s and allow them to assume
any nonnegative real value. By doing so, we shall obtain very good bounds
for $\OPT(s)$ in \secref{sec:relaxation} for every real $s$
(not necessarily integer). Finally, in \secref{sec:sol_integer} we give
the full proof of \thmref{thm:Turan_OPT}.

\subsection{The continuous relaxation}
\label{sec:relaxation}

Following the ideas of Norine in \cite{Nor}, we shall consider a
continuous relaxation of \eqref{opt:rephrased}.
We relax the constraints on the variables $A_i$ by allowing them
to assume any nonnegative real value. In this version of the problem,
we also scale these variables by dividing each $A_i$ by $q$.
The only constraint involving $q$ in \eqref{opt:rephrased} becomes
$\sum_{i=1}^k A_i = 1$, which is now independent of $q$. The other
effect this variable scaling has is of subtracting the value of
the goal function by the constant $\log q$. In addition, it will
be convenient to introduce another parameter $0 \le \delta < \frac{1}{k}$,
which represents the smallest value that one of the variables $\alpha_i$
can assume. The relaxed problem is stated as follows.
\begin{equation}
  \label{opt:relaxed}
  \begin{array}{ll}
    \text{Maximize}   & \sum_{i=1}^k\alpha_i \log A_i\\
    \text{Subject to} & \sum_{i=1}^k\alpha_i=1, \enskip
                        \sum_{i=1}^k\alpha_i^2\le \frac{1}{s},\enskip
                        \sum_{i=1}^k A_i=1,\\
    \text{Where}      &\alpha_i \ge \delta, A_i\ge 0\Text{are real
                       variables, and} k \ge \cei{s}.
  \end{array}
\end{equation}

In above and in the rest of this subsection we do not assume that
$s$ is necessarily integer, and we do allow some of the variables
$A_i$ to assume the value zero. For this, we extend the range of
the goal function to include $-\infty$. This is a minor technical
detail and is used solely to simplify our analysis of the problem.
In that case, we also extend the definition of the goal function as
we set $A\cdot\log B=-\infty$ for $A\ne B=0$ and $A\cdot\log B=0$
for $A=B=0$.

We stress that the optimization problem \eqref{opt:relaxed} is
well-defined, even though the goal function is discontinuous at
the boundary. This is because the goal function is still upper
semi-continuous, and the domain is compact, so the maximum of
\eqref{opt:relaxed} is always attained by a point in the domain.

\medskip

We will prove a sequence of statements about the relaxation,
which will lead us to a complete solution to \eqref{opt:relaxed}.
First, let us determine the values of the $A_i$'s in terms
of the $\alpha_i$'s.

\begin{proposition}

  \label{prop:fixed_alpha}
  For any solution to \eqref{opt:relaxed}, we
  have $A_i = \alpha_i$ for all $i=1,\ldots, k$.

\end{proposition}
\begin{proof}
First observe that the maximum of \eqref{opt:relaxed} is
at least $-\log k$, since we can take $\alpha_i=A_i =\frac{1}{k}$
for all $i=1,\ldots, k$. Moreover, if $A_i = 0$ then $\alpha_i = 0$
because otherwise the value of the goal function would be $-\infty$.
Furthermore, if $\alpha_i = 0$ then we also have $A_i = 0$ because
we could otherwise ``shift the weight'' of $A_i$ to another $A_j$
such that $\alpha_j > 0$ and increase the value of $F(\valpha)$.
Hence, we have $\alpha_i = 0\iff A_i= 0$.

By the method of Lagrange multipliers applied to the variables
$A_i$, there exists $\lambda$ such that $\frac{\alpha_i}{A_i} = \lambda$,
for all $i$ such that $A_i\ne 0$. But the identity $\alpha_i = \lambda A_i$
is still true even if $A_i = 0$ (by the discussion in the previous
paragraph). So $\sum_{i=1}^k \alpha_i = \lambda \cdot\sum_{i=1}^k A_i$.
Therefore $\lambda = 1$, which implies $A_i = \alpha_i$ for all
$i=1,\ldots,k$ and proves the proposition.
\end{proof}

Notice that the equivalent of \propref{prop:tight} still holds
in this context.
\begin{proposition}

  \label{prop:tight_cont}
  For any solution to \eqref{opt:relaxed}, we have
  $\sum_{i=1}^k \alpha_i^2 =\frac{1}{s}$.

\end{proposition}

In view of the previous two propositions we can restate
\eqref{opt:relaxed} as
\begin{equation}
  \label{opt:relaxed2}
  \begin{array}{ll}
    \text{Maximize}   & F(\valpha):=\sum_{i=1}^k\alpha_i \log \alpha_i\\
    \text{Subject to} & \sum_{i=1}^k\alpha_i=1, \enskip
                        \sum_{i=1}^k\alpha_i^2= \frac{1}{s},\\
    \text{Where}      &k \ge \cei{s}, \Text{and} \alpha_i \ge \delta \Text{is real for all} i=1,\ldots, k.
  \end{array}
\end{equation}
In the succeeding proposition we prove an upper bound for $F(\valpha)$
in \eqref{opt:relaxed2}.

\begin{proposition}

  \label{prop:jensen_cont}
  For any $\valpha$ in the domain of \eqref{opt:relaxed2} we
  have $F(\valpha) \le -\log s$. In particular, the maximum
  of \eqref{opt:rephrased} is at most $\log q - \log s$.

\end{proposition}
\begin{proof}
We know that the function $x\mapsto \log x$ is concave for $x>0$,
and by Jensen's inequality we obtain
\[
  F(\valpha)= \sum_{i=1}^k \alpha_i \log \alpha_i \le
  \log\lt(\sum_{i=1}^k \alpha_i^2\rt)= -\log s,
\]
thereby proving the proposition.
\end{proof}

The next proposition reveals further information about the structure
of the solutions of \eqref{opt:relaxed2}.

\begin{proposition}

  \label{prop:two_values}
  If $\valpha = (\alpha_1,\ldots, \alpha_k)$ is a local maximum
  point for \eqref{opt:relaxed2}, then the cardinality of the set
  $\{\alpha_1,\ldots, \alpha_k\}\setminus \{\delta\}$ is at most two.

\end{proposition}
\begin{proof}
We again use the method of Lagrange multipliers.
If $\alpha_i = \frac{1}{k}$ for $i=1,\ldots, k$, the statement of the
proposition is immediately true. Otherwise, $\valpha$ is a regular
point, and thus there exist two multipliers $\lambda, \mu$ such that
\[
  \log \alpha_i + 1 = \mu + \lambda \alpha_i
\]
for all $i$ such that $\alpha_i > \delta$. But the function
$f(x):= \log x - \lambda x +1 - \mu$ is strictly concave regarless of
the values of $\lambda$ and $\mu$, therefore there are at most two roots
of $f(x) = 0$, which proves the proposition.
\end{proof}

It turns out that the case $k=3$ will play vital role in the way
we solve the general case. We thus derive the following two propositions
for this special case. When $k=3$ we have $\alpha_1 + \alpha_2
+ \alpha_3 = 1$ and $\alpha_1^2 + \alpha_2^2 + \alpha_3^2 = \frac{1}{s}$.
It will be convenient to introduce the following parametrization of
the variables
\begin{equation}
  \label{eqn:reparametrization}
  \begin{array}{ll}
  \alpha_1 &= \frac{1}{3} + \frac{\rho}{\sqrt{6}} \cos \theta
    + \frac{\rho}{\sqrt{2}} \sin \theta\\
  \alpha_2 &= \frac{1}{3} + \frac{\rho}{\sqrt{6}} \cos \theta
    - \frac{\rho}{\sqrt{2}} \sin \theta\\
  \alpha_3 &= \frac{1}{3} - \frac{2\rho}{\sqrt{6}} \cos \theta,
  \end{array}
\end{equation}
which clearly satisfies $\alpha_1 + \alpha_2 + \alpha_3 = 1$
and $\alpha_1^2 + \alpha_2^2 + \alpha_3^2 = \frac{1}{s}$,
where $\rho := \sqrt{\frac{3-s}{3s}}$ and $\theta \in [0,2\pi]$
is a new variable. Moreover, any triple $(\alpha_1,\alpha_2, \alpha_3)$
satisfying the constraints of \eqref{opt:relaxed2} can be parametrized as
before. By symmetry, we may even assume $\alpha_1\ge \alpha_2\ge \alpha_3$,
or equivalently, $\theta \in [0, \frac{\pi}{3}]$. The actual range of $\theta$
is an interval of the form $[\theta_0,\frac{\pi}{3}]$, where $\theta_0
\in[0,\frac{\pi}{3}]$ is either the solution of $\alpha_3(\theta_0) = \delta$
if $\frac{1}{3}-\frac{2\rho}{\sqrt{6}}\le \delta$, or $0$ otherwise.

The parametrization \eqref{eqn:reparametrization} allows us to view
$F(\valpha)=F(\alpha_1,\alpha_2,\alpha_3)$ as a function of $\theta$.
With slight abuse of notation, let $F(\theta):= F(\alpha_1(\theta),
\alpha_2(\theta),\alpha_3(\theta))$.
\begin{proposition}

  \label{prop:three_minimum}
  For any $1 < s \le 3$, $\theta=\frac{\pi}{3}$ is a strict
  local minimum point for $F(\theta)$.

\end{proposition}
\begin{proof}
For $1 < s \le 3$, we have $0 \le \rho < \frac{\sqrt{6}}{3}$.
For $\theta=\frac{\pi}{3}$, we clearly have $\alpha_1,\alpha_2,\alpha_3>0$.
Moreover, taking derivatives with respect to $\theta$, we get
\[
  F'(\theta) = \alpha_1' \log \alpha_1 + \alpha_2' \log \alpha_2
    + \alpha_3' \log \alpha_3
\]
and a straightforward computation shows that $F'(\pi/3) = 0$.
Thus, in order to prove that $\frac{\pi}{3}$ is a strict local minimum
point for $F(\theta)$, it is enough to show that the second derivative
of $F(\theta)$ at $\theta=\frac{\pi}{3}$ is positive. Computing
$F''(\theta)$ we obtain
\[
  F''(\theta) = \frac{(\alpha_1')^2}{\alpha_1} +
         \frac{(\alpha_2')^2}{\alpha_2} +
         \frac{(\alpha_3')^2}{\alpha_3} +
         \alpha_1'' \log \alpha_1 +
         \alpha_2'' \log \alpha_2 +
         \alpha_3'' \log \alpha_3
\]
We replace $\theta=\frac{\pi}{3}$ in the identity above and obtain
\begin{align*}
F''(\pi/3) &= \frac{\rho^2}{\frac{1}{3} - \frac{\rho}{\sqrt{6}}}
  -\frac{2\rho}{\sqrt{6}}\log \lt(\frac{\frac{1}{3} + \frac{2\rho}{\sqrt{6}}}{%
  \frac{1}{3}-\frac{\rho}{\sqrt{6}}}\rt) =
  \frac{\rho^2}{\frac{1}{3} - \frac{\rho}{\sqrt{6}}}
    - \frac{2\rho}{\sqrt{6}}\log \lt(1 + \frac{\frac{3\rho}{\sqrt{6}}}{%
    \frac{1}{3}-\frac{\rho}{\sqrt{6}}}\rt).
\end{align*}
Using the inequality $\log(1+x) < x$ we infer that $F''(\pi/3) > 0$,
proving the proposition.
\end{proof}

\begin{proposition}

  \label{prop:cont_3}
  If $\valpha=(\alpha_1,\alpha_2, \alpha_3)$ is a local maximum point for
  \eqref{opt:relaxed2} in the case $k=3$, and $\alpha_1\ge \alpha_2\ge
  \alpha_3 > \delta$, then $\alpha_1 = \alpha_2$.
  Moreover, the function $F(\theta)$ is strictly decreasing in the
  interval $[\theta_0,\frac{\pi}{3}]$.

\end{proposition}
\begin{proof}
Consider any vector $\valpha=(\alpha_1,\alpha_2, \alpha_3)$ achieving
a local maximum of $F$. By \propref{prop:two_values}, we have
$\alpha_1 = \alpha_2$ or $\alpha_2 = \alpha_3$ (since $\alpha_i > \delta$
for all $i=1,2,3$). The case $\alpha_1 > \alpha_2 =
\alpha_3$ corresponds to $\theta = \frac{\pi}{3}$ in the reparametrization
\eqref{eqn:reparametrization}, and by \propref{prop:three_minimum},
$\valpha$ can not be a maximum point for $F$. Therefore, assuming that
$\valpha$ is an optimal solution to \eqref{opt:relaxed2} we conclude that
$\alpha_1 = \alpha_2$ (and thus $\theta=\theta_0 =0$).

Using similar arguments, one can also show that in general
$\theta_0$ and $\frac{\pi}{3}$ are the only two local extremum points of
$F(\theta)$, and thus we have $F(\theta_0)>F(\frac{\pi}{3})$.
Therefore $F(\theta)$ is strictly decreasing in $[\theta_0,\frac{\pi}{3}]$.
\end{proof}

\propref{prop:cont_3} implies the next lemma, which completely solves
the relaxation \eqref{opt:relaxed2} and hence \eqref{opt:relaxed} as well.

\begin{lemma}

  \label{lem:cont_opt}
  If $\valpha = (\alpha_1,\ldots, \alpha_k)$ is a local maximum point
  for \eqref{opt:relaxed2} satisfying $\alpha_1\ge \ldots \ge \alpha_k$
  then there exists an integer $\ell \ge 0$ such that
  $k = \ell + \cei{s^*}$, $\alpha_1=\alpha_2=\cdots=\alpha_{k-\ell-1}\ge
  \alpha_{k-\ell} > \delta$ and $\alpha_{k-\ell+1}=\ldots=\alpha_{k} =\delta$,
  where $s^*=s\frac{(1-\ell\delta)^2}{1-\ell s\delta^2}$.
  In particular, if $\ell = 0$ then $s^*=s$ and $k=\cei{s}$.
  Furthermore, we have
  \begin{equation}
    \label{eqn:bound_cont_opt}
    F(\valpha) \le (1-\ell \delta) \lt[\log(1-\ell\delta) - \log s^*\rt]
     + \ell \delta \log \delta.
  \end{equation}

\end{lemma}
\begin{proof}
Let $\ell$ be the number of indices $i$ such that $\alpha_i = \delta$.
Because we assumed $\alpha_1\ge \alpha_2\ge \ldots \ge \alpha_k$,
we have $\alpha_{k-\ell+1} = \ldots = \alpha_{k} = \delta$.
If we fix all but three variables $\alpha_i$ in \eqref{opt:relaxed2} we obtain
another instance of \eqref{opt:relaxed2} with $k=3$ (up to rescaling of
the variables and parameters $s$ and $\delta$). Hence \propref{prop:cont_3}
can be applied to this sub-problem, and it implies that for any triple
$1 \le i_1 < i_2 <i_3 \le k-\ell$ we have $\alpha_{i_1}=\alpha_{i_2} \ge
\alpha_{i_3}$. From this we infer that $\alpha_1=\ldots =
\alpha_{k-\ell-1} \ge \alpha_{k-\ell}$. Let $x = \frac{\alpha_1}{1-\ell\delta}$
and $y = \frac{\alpha_{k-\ell}}{1-\ell\delta}$. We have $x \ge y >
\frac{\delta}{1-\ell\delta}$, $(k-\ell-1)x + y = 1$ and $(k-\ell-1)x^2 + y^2 =
\frac{1}{s^*}$. This system has a solution provided that $k-\ell \ge s^* >
k-\ell - 1$, or equivalently $k - \ell = \cei{s^*}$. To finish the proof of
the lemma we use \propref{prop:jensen_cont} on the sub-problem restricted
to the first $k-\ell$ variables. The proposition states that
$(k-\ell-1) x \log x + y \log y \le -\log s^*$, which combined
with the identity
\[
F(\valpha) = (1-\ell\delta)\lt[\log(1-\ell\delta) +
  (k-\ell-1) x \log x + y \log y\rt]+\ell \delta \log \delta
\]
implies \eqref{eqn:bound_cont_opt}, and we are done.
\end{proof}

\subsection{Solving the optimization problem for integer $s$}
\label{sec:sol_integer}

We know from \propref{prop:jensen_cont} that any solution of
$\OPT(s)$ has objective value at most $\log q-\log s$, whenever
$s$ is an integer number. But how close to this bound is the
$s$-balanced vector for $q$? The answer to this question is
contained in the following proposition, which shall be used as
a benchmark for comparison with other candidate solutions.

\begin{proposition}

  \label{prop:good_approx}
  For $s$ integer and $q\ge s$, if $\valpha$ is the $s$-balanced
  vector for $q$ then
  \[
    \OBJ(\valpha) \ge \log q - \log s -\frac{s^2}{2q^2}.
  \]

\end{proposition}
\begin{proof}
We divide $q$ by $s$ (with remainder) as $q = st + r$ where $t,r$ are
integers and $0 \le r < s$. Note that $t\ge 1$ because $q\ge s$.
If we denote by $\valpha$ the $s$-balanced vector for $q$ then
\[
  \OBJ(\valpha) = \frac{r}{s} \log(t+1) + \frac{s-r}{s} \log t.
\]
Let $f(x):= \log x$. If $x\in[t,t+1]$ then $-\frac{1}{t^2} \le f''(x)
\le -\frac{1}{(t+1)^2}$, so the function $g(x):=\log x + \frac{x^2}{2t^2}$
is convex. By Jensen's inequality, we have $g(\frac{q}{s}) \le
\frac{1}{s}\sum_{i=1}^s g(A_i)$ for any balanced partition $A_1,\ldots,A_s$ of $[q]$, 
which implies
\[
  \OBJ(\valpha) \ge \log q - \log s - \frac{1}{2t^2}
  \cdot \frac{r(s-r)}{s^2} \ge \log q - \log s - \frac{s^2}{2q^2},
\]
thereby proving the proposition.
\end{proof}

\bigskip

Throughout the remaining of the section, the vectors $\valpha=(\alpha_1,
\ldots,\alpha_k, A_1,\ldots,A_k)$ we consider are from the domain of \eqref{opt:rephrased}.  
The following list contains the proposed steps towards the proof
of \thmref{thm:Turan_OPT}:
\begin{enumerate}[(I) -]
\item If $\valpha=(\alpha_1,\ldots,\alpha_k,
A_1, \ldots, A_k)$ is a solution of $\OPT(s)$ then either $k=s$
or $k=s+1$. Moreover if $k=s+1$ then $\alpha_k$ is small and $A_k=1$.
This statement is a consequence of the following.
\begin{enumerate}[({I}-1) -]
\item A discrete analog of \propref{prop:fixed_alpha} holds,
that is, $A_i \approx \alpha_i \cdot q$ for all $i$.
\item By using the continuous relaxation of \eqref{opt:rephrased}
from \secref{sec:relaxation} we prove that if $k\ge s+1$ then
$\alpha_k$ is tiny. Moreover, if $k\ge s+2$ then both $\alpha_{k-1}$
and $\alpha_k$ are small.
\item If both $\alpha_{k-1}$ and $\alpha_k$
are sufficiently small, then $\valpha$ cannot be a solution of
$\OPT(s)$.
\end{enumerate}
\item If $\valpha=(\alpha_1,\ldots,\alpha_{s+1},
A_1, \ldots, A_{s+1})$ is such that $A_{s+1} = 1$, then we can estimate
\[
  \OBJ(\valpha) \le S_1 + \frac{S_2-S_1^2}{2S_1},
\]
where $S_1=\frac{1}{s}\sum_{i=1}^s \log A_i$ and $S_2=\frac{1}{s}
\sum_{i=1}^s\log^2 A_i$. Moreover, if $\alpha_{s+1}\le \frac{1}{50q}$
then $S_2 - S_1^2 \le \frac{S_1^2}{20q}$.
\item The expression $S_1 + \frac{S_2-S_1^2}{2S_1}$ in step (II)
is maximized when $(A_1,\ldots, A_s)$ forms a balanced partition of
$q-1$.
\item The $s$-balanced vector is better than any candidate solution
with $k=s+1$ satisfying the conditions stated in step (I).
\end{enumerate}

Let us begin with discrete equivalent of \propref{prop:fixed_alpha}
mentioned in step (I-1).

\begin{proposition}

  \label{prop:ratio_alpha_A}
  Let $\valpha=(\alpha_1,\ldots,\alpha_k, A_1,\ldots, A_k)>0$ be a
  solution of $\OPT(s)$, where $s$ is any real parameter. For each
  $i$, we have $|A_i - \alpha_i q|<1+(k-2)\alpha_i$.
  In particular, $\alpha_i \ge \frac{1}{2q}$ whenever $A_i\ge 2$.

\end{proposition}
\begin{proof}
Fix an arbitrary integer $i$. We shall prove that $\alpha_i A_j-
\alpha_j A_i > -\alpha_i - \alpha_j$ for every other $j\ne i$.
This inequality is trivially true when $A_i=1$,
so we may assume that $A_i \ge 2$. For each value of $j\ne i$
we have $\alpha_i \log(A_i - 1) + \alpha_j \log(A_j + 1) \le \alpha_i \log A_i +
\alpha_j \log A_j$. This is because $\valpha$ is a solution of $\OPT(s)$
and thus $\OBJ(\valpha) \ge \OBJ(\valpha')$, where $\valpha'$
is obtained from $\valpha$ by replacing $A_i$ with $A_i-1$
and $A_j$ with $A_j + 1$. Since $\log(1+x)<x$ for all $x>-1$, we have
\[
  0 \le \alpha_i \log\lt(1+\frac{1}{A_i-1}\rt) +
  \alpha_j \log\lt(1-\frac{1}{A_j+1}\rt) <
  \frac{\alpha_i}{A_i-1} - \frac{\alpha_j}{A_j+1},
\]
which implies $\alpha_i A_j -\alpha_j A_i > -\alpha_i - \alpha_j$.
By switching the roles of $i$ and $j$ we also obtain
$\alpha_j A_i -\alpha_i A_j > -\alpha_i -\alpha_j$, which
implies $|\alpha_i A_j -\alpha_j A_i| < \alpha_i + \alpha_j$.

Adding up these inequalities for all $j\ne i$, we obtain
$|\alpha_i(q - A_i) - (1-\alpha_i)A_i| < 1 +(k-2)\alpha_i$,
thereby finishing the proof of the proposition.
\end{proof}

We turn to step (I-2) which can is stated in the following lemma.
\begin{lemma}

  \label{lem:small_two}
  Let $\valpha=(\alpha_1,\ldots,\alpha_k, A_1,\ldots, A_k)>0$ be a
  solution of $\OPT(s)$ for large enough integer $s$, where
  $\alpha_1\ge \alpha_2\ge \ldots \ge \alpha_k$ and $q \ge 100
  \frac{s^2}{\log s}$. If $k \ge s + 1$ then $\alpha_{k} < \frac{1}{50q}$.
  Moreover, if $k \ge s + 2$ then $\alpha_{k-1} < \frac{1}{40q}$.

\end{lemma}

Before proving \lemref{lem:small_two}, we need the following corollary
of \lemref{lem:cont_opt}.

\begin{lemma}

  \label{lem:small_alphak}
  Let $\valpha=(\alpha_1,\ldots,\alpha_k, A_1,\ldots, A_k)>0$ be any
  element in the feasible set $\FEAS(s)$  for large enough $s$
  (not necessarily integer). If $k > \cei{s}$ and $\alpha_1\ge
  \ldots \ge \alpha_k \ge \frac{1}{50q}$ then
  \[
    \OBJ(\valpha) < \log q - \log s -
    \frac{1}{150q} \log\lt(\frac{50q}{s}\rt).
  \]

\end{lemma}
\begin{proof}
Clearly $k \le q$. We consider the continuous optimization problem
\eqref{opt:relaxed} for $s$, $k$, and $\delta:= \frac{1}{50q}$,
and we apply \lemref{lem:cont_opt} to this particular instance.
Since $k > \cei{s}$ we have $\ell > 0$. Moreover, because $\ell,s\le q$,
we have $(50q-\ell)^2 < 2500q^2 -s\ell$ and thus
$s^*=s\cdot \frac{(50q-\ell)^2}{2500q^2 -s\ell} < s$.
Note that $\valpha$ is an element of the domain of \eqref{opt:relaxed}
for our choice of parameters, and by \eqref{eqn:bound_cont_opt} we have
\begin{align*}
  \OBJ(\valpha) &\le \log q+F(\valpha)\le \log q - \lt(1-\frac{\ell}{50q}\rt)\log s^*
  - \frac{\ell \log(50q)}{50q} \\
  &\le \log q - \log s + \log\lt(\frac{2500q^2 - s\ell}{(50q-\ell)^2}\rt)
  -\frac{\ell}{50q}\log\lt(\frac{50q}{s}\rt).
\end{align*}
The function $x\mapsto \log\lt(\frac{2500q^2 - sx}{(50q-x)^2}\rt)
-\frac{x}{50q}\log\lt(\frac{50q}{s}\rt)$ is decreasing
for $1 \le x \le k$, since its derivative is
\[
  -\frac{s}{2500q^2-sx} + \frac{2}{50q-x} -
  \frac{1}{50q}\log\lt(\frac{50q}{s}\rt) <
  \frac{2}{49q} - \frac{\log 50}{50q} < 0,
\]
therefore we conclude that
\begin{align*}
  \OBJ(\valpha) + \log s -\log q &\le \log\lt(1+\frac{100q-s-1}{(50q-1)^2}\rt)
  - \frac{1}{50q}\log\lt(\frac{50q}{s}\rt) \\
  &< \frac{100q}{(50q-1)^2}-
   \frac{1}{50q}\log\lt(\frac{50q}{s}\rt) <
  -\frac{1}{150q}\log\lt(\frac{50q}{s}\rt)
\end{align*}
finishing the proof.
\end{proof}
Equipped with \lemref{lem:small_alphak} we can now prove \lemref{lem:small_two}.

\begin{proof}[Proof of \lemref{lem:small_two}]
Let us first prove that for $k\ge s + 1$ we have $\alpha_k < \frac{1}{50q}$.
Suppose, for contradiction, that $k \ge s+1$ and $\alpha_k\ge\frac{1}{50q}$.
By \lemref{lem:small_alphak}, we have
\[
  \OBJ(\valpha) < \log q - \log s - \frac{1}{150q}\log\lt(\frac{50q}{s}\rt),
\]
but this is a contradiction, since \propref{prop:good_approx} implies that
$\OBJ(\valpha)=\OPT(s) \ge \log q - \log s - \frac{s^2}{2q^2}$
and $\frac{1}{150q}\log\lt(\frac{50q}{s}\rt) > \frac{s^2}{2q^2}$. For the
second part suppose, towards contradiction, that $k\ge s + 2$ and
$\alpha_{k-1} \ge \frac{1}{40q}$. Since $\alpha_k < \frac{1}{50q}$,
from \propref{prop:ratio_alpha_A} we infer that $A_k = 1$. 
Let $\alpha_i':=\frac{\alpha_i}{1-\alpha_k}$ for $i=1,\ldots, k-1$, let $\valpha':=(\alpha_1',
\ldots, \alpha_{k-1}', A_1,\ldots, A_{k-1})$ and let $s':=s\cdot\frac{
(1-\alpha_k)^2}{1-s\alpha_k^2}$. Clearly $s'<s$, $\sum_{i=1}^{k-1}\alpha_i' = 1$
and $\sum_{i=1}^{k-1}(\alpha_i')^2 = \frac{1}{s'}$, hence $\valpha'\in
\FEASq{q-1}(s')$. Moreover $\OBJ(\valpha) = (1-\alpha_k)\OBJq{q-1}(\valpha')$.
Furthermore, since $\alpha_{k-1}'>\frac{1}{40q} >\frac{1}{50(q-1)}$ and
$k-1 \ge s+1>s'+1$ we can apply \lemref{lem:small_alphak} and deduce
\[
  \OBJq{q-1}(\valpha') < \log (q-1) - \log s' - \frac{1}{150(q-1)}
  \log \lt(\frac{50(q-1)}{s'}\rt).
\]
This together with $\log s' = \log s + \log\lt(1-\frac{2\alpha_k-(s+1)
\alpha_k^2}{1-s\alpha_k^2}\rt)> \log s - 3\alpha_k$ implies that
\begin{align*}
  \OBJ(\valpha) \le \OBJq{q-1}(\valpha')&\le
  \left(\log q-\frac{1}{2q}\right)-\log s+3\alpha_k-\frac{1}{150(q-1)}
  \log \lt(\frac{50(q-1)}{s}\rt)\\
  &< \log q-\log s-\frac{1}{150(q-1)}  \log \lt(\frac{50(q-1)}{s}\rt),
\end{align*}
which gives us a contradiction as before.
\end{proof}

The following lemma establishes step (I) by combining the steps
(I-1) and (I-2) together with the proof of (I-3).

\begin{lemma}

  \label{lem:k_or_kplus1}
  Let $\valpha=(\alpha_1,\ldots,\alpha_k, A_1,\ldots, A_k)>0$ be a
  solution of $\OPT(s)$ for large enough integer $s$. If
  $q\ge 100\frac{s^2}{\log s}$, then we must have either
  $k = s$ or $k = s + 1$.

\end{lemma}
\begin{proof}
Suppose, towards contradiction, that $k \ge s+2$. We assume, without loss
of generality, that $\alpha_1\ge \ldots \ge \alpha_k$. By
\lemref{lem:small_two} we must have $\alpha_k,\alpha_{k-1} < \frac{1}{40q}$,
and $\alpha_1 \ge \frac{1}{k}\ge \frac{1}{q}\ge 4\alpha_{k-1}$.
Moreover, by \propref{prop:ratio_alpha_A}, we infer that $A_k = A_{k-1}=1$,
and $\alpha_1\ge \frac{A_1}{4q}$.
As $(\alpha_1 - \alpha_k - \alpha_{k-1})^2\ge 4\alpha_k \alpha_{k-1}$,
we may define $\zeta:= \frac{(\alpha_1 - \alpha_k - \alpha_{k-1}) - \sqrt{
(\alpha_1 - \alpha_k - \alpha_{k-1})^2 - 4\alpha_k \alpha_{k-1}}}{2}$
such that the following holds
\[
 (\alpha_1-\zeta)^2 + (\alpha_k + \alpha_{k-1}+\zeta)^2=
 \alpha_1^2 + \alpha_k^2 + \alpha_{k-1}^2.
\]
We claim that
$\alpha_1 \log A_1 < (\alpha_1 - \zeta) \log A_1 +
(\alpha_k + \alpha_{k-1} + \zeta) \log 2$. To see this,
first note that
\[
 \zeta = \frac{2\alpha_k\alpha_{k-1}}{
 (\alpha_1 - \alpha_k - \alpha_{k-1})+
 \sqrt{(\alpha_1 - \alpha_k - \alpha_{k-1})^2 - 4\alpha_k \alpha_{k-1}}}
  \le \frac{4\alpha_k\alpha_{k-1}}{\alpha_1} \le
  \frac{4\alpha_k\alpha_{k-1}\cdot 4q}{A_1}\le
  \frac{2}{5}\cdot\frac{\alpha_k}{A_1},
\]
hence $\zeta \log A_1 \le \frac{2}{5}\cdot\frac{\log A_1}{A_1}\cdot
\alpha_k\le \frac{2}{5}\alpha_k<(\alpha_k + \alpha_{k-1} + \zeta)\log 2$,
proving our claim. But this is a contradiction, because $\valpha':=
(\alpha_1-\zeta,\alpha_2,\ldots, \alpha_{k-2},\alpha_{k-1} +
\alpha_k + \zeta, A_1, A_2,\ldots, A_{k-2}, 2)$ is a feasible
solution to $\OPT(s)$ yielding a better objective value than
$\valpha$, therefore $k \le s+1$.
\end{proof}

Step (II) is a consequence of the following lemma.
\begin{lemma}

  \label{lem:objective}
  If $\valpha=(\alpha_1,\ldots,\alpha_{s+1}, A_1, \ldots, A_{s+1})$ is a
  solution to $\OPT(s)$ such that $A_{s+1} = 1$ then
  \begin{equation}
    \label{eqn:objective}
    \OBJ(\valpha) \le S_1 + \frac{S_2-S_1^2}{2S_1},
  \end{equation}
  where $S_1=\frac{1}{s}\sum_{i=1}^s \log A_i$ and $S_2=\frac{1}{s}
  \sum_{i=1}^s\log^2 A_i$. If, in addition, $\alpha_{s+1}\le \frac{1}{50q}$
  then $S_2 - S_1^2 \le \frac{S_1^2}{20q}$.

\end{lemma}
\begin{proof}
We fix $A_1,\ldots, A_{s+1}$ and recall \lemref{lem:fixed_Pk}. It states
that for this fixed sequence $A_1,\ldots, A_{s+1}$, we have
$\alpha_i = \frac{\log A_i - \mu}{\lambda}$ for $i=1,\ldots, s+1$
and $\OBJ(\valpha) = \mu + \frac{\lambda}{s}$, where
\[
  \mu = S_1' - \frac{\lambda}{s+1},\enspace
  \lambda = (s+1)\sqrt{s \lt(S_2' - (S_1')^2\rt)},\enspace
  S_1'=\frac{1}{s+1}\sum_{i=1}^{s+1} \log A_i,
  \enspace\text{and}\enspace
  S_2'=\frac{1}{s+1}\sum_{i=1}^{s+1} \log^2 A_i.
\]
We have $S_1'=\frac{s}{s+1}S_1$ and $S_2' =\frac{s}{s+1} S_2$, hence
\begin{equation}
  \label{eqn:lambda}
  \lambda = (s+1) \sqrt{\frac{s^2}{s+1}(S_2-S_1^2)+
  \frac{s^2}{(s+1)^2}S_1^2}.
\end{equation}
We estimate $\lambda$ in \eqref{eqn:lambda} using the inequality
$\sqrt{a^2 + b} \le a + \frac{b}{2a}$ for $a := \frac{s}{s+1}S_1$
and $b:=\frac{s^2}{s+1} (S_2-S_1^2)$ and obtain
\[
  \lambda \le s\cdot S_1 + \frac{s(s+1)}{2 S_1}\cdot(S_2-S_1^2).
\]
Using this bound for $\lambda$ in the formula $\OBJ(\valpha)=
\frac{s}{s+1}S_1 +\frac{\lambda}{s(s+1)}$ we obtain
\eqref{eqn:objective}. To finish the proof of the lemma, we need
to show that $S_2-(S_1)^2\le \frac{1}{20q}S_1^2$ whenever $\alpha_{s+1}
\le \frac{1}{50q}$. For that purpose, we use the identity $\alpha_{s+1}
= \frac{\log A_{s+1}-\mu}{\lambda}=\frac{\lambda-sS_1}{(s+1)\lambda}$,
which, together with the inequality $0\le \alpha_{s+1} \le \frac{1}{50q}$,
implies 
\[
  \lambda\le \frac{3sS_1}{2} \Text{~~and~~} 0 \le \lambda - s S_1 \le \frac{s(s+1)}{30q} S_1.
\]
Multiplying the above inequality by $\lambda+sS_1$, we obtain
$S_2-(S_1)^2\le \frac{1}{20q}S_1^2$, which completes the proof.
\end{proof}

With $S_1$ and $S_2$ as in the statement of \lemref{lem:objective}
we have step (III).

\begin{lemma}

  \label{lem:objective2}
  For $s$ large enough and $q\gg s$, the maximum of
  \[
    S_1 + \frac{S_2 - S_1^2}{2S_1},
  \]
  over all choices of $(A_1,\ldots, A_{s+1})$ satisfying both
  $A_{s+1} = 1$ and $S_2-S_1^2 \le \frac{S_1^2}{20q}$,
  is attained when $(A_1,\ldots, A_{s})$ forms a balanced partition
  of $q-1$.

\end{lemma}
\begin{proof}
First we claim that under the conditions of the lemma, we have
$A_i = (1 + o(1)) \frac{q}{s}$ for all $i=1,\ldots, s$, and so
$S_1=(1+o(1))\log \frac{q}{s}$, where the asymptotic notation symbol
$o(1)$ represents a function that tends to zero as $s$ tends to
infinity. This is because $S_1 \le \log \frac{q}{s}$
(by the concavity of $\log$) and the variance $S_2 -S_1^2$ can be
bounded below by
$\frac{1}{4s}(\log A_i - \log A_j)^2$ for any $i\ne j$, and thus
\[
  (\log A_i  -\log A_j) ^2 \le 4s\cdot\frac{S_1^2}{20q} = O\lt(
  \frac{\log^2 (q/s)}{q/s}\rt) = o(1).
\]

Suppose, towards contradiction, that the lemma is false, and let
$1\le l, m \le s$ be such that $A_{l} \ge A_{m} + 2$. Let
$\widetilde A_l:=A_l-1$, $\widetilde A_m:=A_m+1$, and $\widetilde A_i:=A_i$
for all $i\not\in \{l,m\}$.
Moreover, let $\widetilde S_1 = \frac{1}{s} \sum_{i=1}^s \log \widetilde A_i$
and $\widetilde S_2 = \frac{1}{s}\sum_{i=1}^s \log^2 \widetilde A_i$. To
arrive at a contradiction, it suffices to show that
$\widetilde S_1 +\frac{\widetilde S_2-\widetilde S_1^2}{2\widetilde S_1}
>S_1+\frac{S_2-S_1^2}{2S_1}$.
Since $\log(1+x)=x + O(x^2)$, we have
\begin{equation}
  \label{eqn:ineq_S1}
  \widetilde S_1 -S_1 = \frac{1}{s} \log\lt(1+\frac{A_l-A_m-1}{A_lA_m}\rt)
  = \lt(1+o(1)\rt)\cdot \frac{A_l-A_m-1}{sA_lA_m}.
\end{equation}
Moreover, if $f(x):= \log^2(x+1) - \log^2 x$, we have
$\widetilde S_2 -S_2 = \lt(f(A_m) - f(A_l-1)\rt) / s$. By the mean value
theorem, there exists $B$ such that $A_m \le B \le A_l - 1$ such
that $\widetilde S_2-S_2 =\frac{A_m-A_l+1}{s} f'(B)$. However, since
$f'(B) = 2\cdot\lt[\frac{\log(B+1)}{B+1}-\frac{\log B}{B}\rt]$,
another application of the mean value theorem yields $f'(B)=
-2\cdot \frac{\log(C/e)}{C^2}$ for some $B \le C \le B+1$, hence
$A_m\le C\le A_l$. These two identities imply
\begin{equation}
  \label{eqn:ineq_S2}
  \widetilde S_2 -S_2=2\cdot\frac{A_l-A_m-1}{s} \cdot \frac{\log(C/e)}{C^2}
  \ge \frac{3}{2}\cdot\frac{(A_l-A_m-1)\log \frac{q}{s}}{sA_lA_m}.
\end{equation}
Combining \eqref{eqn:ineq_S1} and \eqref{eqn:ineq_S2} yields
\[
  \widetilde S_1+\frac{\widetilde S_2 -\widetilde S_1^2}{2\widetilde S_1}
  \ge S_1 + \frac{S_2 - S_1^2}{2S_1}+\frac{(A_l-A_m-1)}{2sA_lA_m},
\]
which is a contradiction, thereby proving our claim that
$|A_l - A_m| \le 1$ for all $1\le l,m \le s$.
\end{proof}

We are ready to prove step (IV) and hence \thmref{thm:Turan_OPT}.

\begin{proof}[Proof of \thmref{thm:Turan_OPT}]
As mentioned in the beginning of this section, in order to prove
\thmref{thm:Turan_OPT} it suffices to show that for any solution
$\valpha$ of $\OPT(s)$, the support graph $\SUPP(\valpha)$ lies
in $\mc{P}_s$. On the other hand, we already know (by
\lemref{lem:k_or_kplus1}) that $\SUPP(\valpha)\in \mc{P}_s\cup
\mc{P}_{s+1}$. Suppose, towards contradiction, that $\SUPP(\valpha)
\in \mc{P}_{s+1}$, where $\valpha=(\alpha_1,\ldots,\alpha_{s+1},
A_1,\ldots, A_{s+1})$. By \lemref{lem:small_two}, we know that
$\alpha_{s+1} \le \frac{1}{50q}$ and as a consequence of
\propref{prop:ratio_alpha_A} we obtain $A_{s+1} = 1$.

From \lemref{lem:objective} we obtain
\[
  \OBJ(\valpha) \le S_1 + \frac{S_2 - S_1^2}{2S_1} \Text{~~and~~} S_2-S_1^2\le \frac{S_1^2}{20q},
\]
where $S_1= \frac{1}{s}\sum_{i=1}^s \log A_i$ and $S_2=\frac{1}{s}
\sum_{i=1}^2 \log^2 A_i$. By \lemref{lem:objective2}, we infer that
that the maximum of $S_1 + \frac{S_2 - S_1^2}{2S_1}$ is attained
when $(A_1,\ldots, A_s)$ is a balanced partition of $q-1$. 
For convenience of the remaining proof, we may assume that $(A_1,\ldots, A_s)$ is a balanced partition of $q-1$.
Then we have $|A_i-A_j| \le 1$ for all $1\le i, j\le s$ and thus $|A_i - \exp(S_1)|\le 1$. This
implies that $2S_1\ge \log \frac{q}{s}$ and $\log A_i - S_1 \le \frac{2s}{q}$, 
the latter one of which in turn implies that
\[
  S_2 - S_1^2 \le \frac{4s^2}{q^2}.
\]
Let $S^*$ be the value of the objective function on the $s$-balanced
vector for $q$. We have
\[
  S^* \ge S_1 + \frac{1}{s}\log\lt(1+\frac{1}{A_1}\rt)
  \ge S_1 +\frac{1}{2q},
\]
but
\begin{align*}
  S^*\le \OBJ(\valpha) \le S_1
  + \frac{S_2 - S_1^2}{2S_1} \le S_1+\frac{4s^2}{q^2\log\frac{q}{s}}
  \le  S_1 + \frac{1}{10q},
\end{align*}
and this final contradiction finishes the proof.
\end{proof}

%% file: conclusion.tex
\section{Concluding remarks and open problems}
\label{sec:conclusion}
In \thmref{thm:structure} we show that for any real $s>1$ and integer $q$,
the support graph of every solution to $\OPT(s)$ is in either
$\cup_{\cei{s}\le k\le q} \mc{P}_k$ or $\mc{Q}_{\cei{s}}$. We point out
that the family $\mc{Q}_{\cei{s}}$ can not be further reduced.
This can be demonstrated by the results from \cite{LPS}. The authors
of \cite{LPS} proved that when the edge density is smaller than $1/(q\log q)$
(so the corresponding $s$ is less than 2), the extremal graphs which maximize
the number of $q$-colorings are some complete bipartite graphs
plus isolated vertices, which corresponds to the family $\mc{Q}_2$.

One of the reasons that we were able to solve $\OPT(s)$ for integers $s$
in \secref{sec:opt_integer_s} is that the family $\mc{Q}_\cei{s}$ vanishes
when $s$ is integer.
In general, it remains difficult to solve the $\OPT(s)$ for every $q$.
However, we wonder if the following statement is true for any real $s>1$:
there exists a function $q(s)$ such that for any integer $q\ge q(s)$,
every extremal graph with sufficiently large $n$ vertices and
$m=\frac{s-1}{2s}n^2$ edges maximizing the number of $q$-colorings is
$o(n^2)$-close to (or even is) a complete $\cei{s}$-partite graph.
Equivalently, it says that under the same conditions, every solution to
$\OPT(s)$ is in $\mc{P}_{\cei{s}}$. \thmref{thm:main} shows that this
holds for all large integers $s$.

We can improve \thmref{thm:main} to that every extremal graph is
$O_{s,q}(n)$-close to the Tur\'an graph $T_s(n)$, where the dependence
in $O$ is relative to $s$ and $q$. The proof requires lengthy and tedious
stability arguments and we decide to not include here. The problem of
pursuing the exact structure of the extremal graphs in fact can be reduced
to a universal maximum bound on $P_G(q)$ for all sparse graphs $G$ and
general $q$, which we explain as follows. Let $G$ be a graph with $n$
vertices and $m=o(n^2)$ edges. A result from \cite{LPS} asserts that when
$q$ is fixed and $m$ is sufficiently large (so is $n$), the maximum of
$P_G(q)$ is  $q^n\cdot e^{\lt((-c+o(1))\sqrt{m}\rt)}$,
where $c=2\sqrt{\log\frac{q}{q-1}\cdot \log q}$. And when $n,m$ are fixed
and $q$ is sufficiently large, it is not hard to see that the maximum of
$P_G(q)$ is at most $q^n\cdot\lt(1-\frac{1}{q}\rt)^{c'm}$ for some absolute
constanct $c'$. However it is not known if the following universal upper
bound $P_G(q)\le \max\lt\{q^n\cdot e^{\lt((-c+o(1))\sqrt{m}\rt)}, ~
q^n\cdot\lt(1-\frac{1}{q}\rt)^{c'm}\rt\}$
or similar holds for all such sparse graphs $G$ and for general $m,n,q$.
If the answer is yes, this would lead to the exact structure of the extremal
graphs, which are the Tur\'an graphs. Otherwise, there exists a sparse $G$
with a larger number of $q$-colorings in some range of $q$; adding this $G$
to certain $s$-partite graph will likely give a counterexample to
\conjref{conj:Lazebnik} in that range of $q$.

Another related question, raised in \cite{MT} (also see \cite{Lin}),
was asked to find the maximum number of acyclic orientations, that is
the value of $|P_G(-1)|$, over graphs $G$ with $n$ vertices and $m$ edges.
An upper bound was obtained in \cite{Lin} that it is at most
the product of $\max\{2,d(x)\}$ over all vertices $x$, where
$d(x)$ denotes the degree of $x$. It will be also interesting
to find the extremal graphs in this context.

We also feel that the problem we study, maximizing the number of proper
$q$-colorings over graphs with fixed number of vertices and edges, shares
certain similarity with the result of Reiher \cite{Reiher} which finds the
minimum number of cliques $K_q$ over the same family of graphs. An evidence
is that the solutions to the continuous relaxation \eqref{opt:relaxed} are
very similar to the extremal graphs Reiher found.

In \cite{LPS}, Loh {\it et al.} remarked that ``the natural next step
would be to extend the result to the range $\frac{m}{n^2}\le \frac{1}{4}$''
for general $q$. That is the case $1<s\le 2$. We will address this in a
forthcoming paper.

\medskip

\noindent\textbf{Acknowledgement.} The authors wish to thank Benny Sudakov for suggesting 
the problem to them, and Shagnik Das for invaluable discussions and specially for pointing out that the
counterexamples can still be found in a wider range of $q$ than what was originally pursued.

%% file: appendix.tex
\appendix

\section{Stability of the main optimization problem}
\label{sec:stability}

In this section, we prove \corref{cor:stability-LPS}.

\begingroup
\def\thetheorem{\ref{cor:stability-LPS}}
\begin{corollary}

  For any real $s>1$, the following holds for all sufficiently large $n$.
  Let $G$ be an $n$-vertex graph with $m=\frac{s-1}{2s}n^2 + o(n^2)$ edges
  which maximizes the number of $q$-colorings. Then $G$ is $o(n^2)$-close
  to $G_{\valpha}(n)$ for some $\valpha$ which solves $\OPT(s)$.

\end{corollary}
\addtocounter{theorem}{-1}
\endgroup
\begin{proof}
Suppose, towards contradiction, that the corollary is false.
That means that there exist $\eps>0$ and a sequence of graphs
$\{G_t\}_{t=1}^\infty$ such that
\begin{itemize}
\item $G_t$ has $n_t$ vertices and $m_t$ edges;
\item $\lim_{t\to\infty} n_t=\infty$ and $m_t=\frac{s-1}{2s}n_t^2+o_t(n_t^2)$;
\item $G_t$ maximizes the number of $q$-colorings among all graphs with
the same number of vertices and edges;
\item $G_t$ is not $\eps n_t^2$-close to $G_{\valpha}(n_t)$ for
any solution $\valpha$ of $\OPT(s)$.
\end{itemize}
We may assume, by possibly passing to a subsequence of
$\{G_t\}_{t=1}^{\infty}$, that for all $t\ge 1$ we have
\[
 m_t \le \frac{(s+1)-1}{2(s+1)}n_t^2
\]
and that $n_t$ is large enough so that we can apply
\thmref{thm:stability-LPS} to the graph $G_t$ with $s$ replaced by
$s+1$ and with $\eps$ replaced by $\frac{\eps}{2t}$. This implies that
there exists $s_t$ and $\valpha_t$ such that the following holds:
\begin{itemize}
\item $\lt|\frac{s_t-1}{2s_t}-\frac{m_t}{n_t^2}\rt| < \frac{\eps}{2t}$
 and $s_t \le s+1$;
\item $\valpha_t$ is a solution of $\OPT(s_t)$;
\item $G_t$ is $\frac{\eps}{2t}n_t^2$-close to $G_{\valpha_t}(n_t)$;
\end{itemize}
Note that in this case we have $\lim_{t\to\infty} \lt|\frac{s_t-1}{2s_t}
-\frac{m_t}{n_t^2}\rt| = 0$, which implies $\lim_{t\to\infty} s_t = s$.

Again, by possibly passing to a subsequence, we may assume that
the sequence $\valpha_t$ (which lives in a compact space) converges
to some $\valpha$. By the continuity of $\V$ and $\E$ we have
\[
  \V(\valpha)=\lim_{t\to\infty} \V(\valpha_t) = 1\qquad\text{and}\qquad
  \E(\valpha)=\lim_{t\to\infty} \E(\valpha_t) = \lim_{t\to\infty}
  \frac{s_t-1}{2s_t} = \frac{s-1}{2s},
\]
hence $\valpha\in \FEAS(s)$. Furthermore, by the continuity of
$\OPT$ and $\OBJ$, we have
\[
  \OPT(s) = \lim_{t\to\infty} \OPT(s_t) = \lim_{t\to\infty} \OBJ(\valpha_t)
  = \OBJ(\lim_{t\to\infty} \valpha_t) = \OBJ(\valpha),
\]
hence $\valpha$ is a solution to $\OPT(s)$. Lastly, since $\valpha_t\to
\valpha$, for $t$ sufficiently large we have that $G_{\valpha}(n_t)$
and $G_{\valpha_t}(n_t)$ are $\frac{\eps}{2} n_t^2$-close.
But because $G_t$ is $\frac{\eps}{2t}n_t^2$-close to $G_{\valpha_t}(n_t)$,
we have that $G_t$ is $\eps n_t^2$-close to $G_{\valpha}(n_t)$, a contradiction.
This concludes the proof of the corollary.
\end{proof}